\newif\ifsiam\siamfalse

\PassOptionsToPackage{usenames,dvipsnames,table}{xcolor}
\ifsiam
	\documentclass[review,onefignum,onetabnum]{siamart190516}
\else
	\documentclass[USenglish,10pt,a4paper]{article}
\fi
	\PassOptionsToPackage{capitalize,nameinlink}{cleveref}
	\PassOptionsToPackage{breakspaces}{clevethm}
	\usepackage{centernot}
	\usepackage{tcolorbox}

	\ifsiam
		\usepackage[
			customsettings=false,
			noload={amsthm,txfonts},
			excludethm=all,
		]{myPreamble}
	\else
		\usepackage[bottom=1.2in,left=1.1in,right=1.1in]{geometry}
		\usepackage[
			noqed,
			customsettings=false,
			noload={txfonts},
			eqreset=section,
			thmreset=section,
		]{myPreamble}
		\hypersetup{%
			colorlinks = true,
			urlcolor   = BlueIMT,
			linkcolor  = MidnightBlue,
			citecolor  = Green,
		}
	\fi
	\pgfplotsset{compat=1.16}

	\ifsiam
		\renewenvironment{qedequation*}{\[}{\]}
	\else
		\renewcommand{\includetikz}[2][]{\includegraphics[#1]{Pics/Tikz/#2.pdf}}
	\fi
	\ifsiam
		\newsiamremark{remark}{Remark}
		\newsiamremark{example}{Example}

		\theoremstyle{plain}
		\theoremheaderfont{\normalfont\itshape}
		\theorembodyfont{\normalfont}
		\theoremseparator{.}
		\theoremsymbol{}
		\newtheorem{assumption}{Assumption}

		\CreateTheoremLists{corollary}
		\CreateTheoremLists{definition}
		\CreateTheoremLists{example}
		\CreateTheoremLists{fact}
		\CreateTheoremLists{lemma}
		\CreateTheoremLists{proposition}
		\CreateTheoremLists{remark}
		\CreateTheoremLists{theorem}
		
		\makeatletter
			\let\oldproposition\proposition
			\let\oldendproposition\endproposition
			\renewenvironment{proposition}[1][]{%
				\def\@currentlabelname{#1}\ifstrempty{#1}{\oldproposition}{\oldproposition[#1]}%
			}{\oldendproposition}
			\let\oldlemma\lemma
			\let\oldendlemma\endlemma
			\renewenvironment{lemma}[1][]{%
				\def\@currentlabelname{#1}\ifstrempty{#1}{\oldlemma}{\oldlemma[#1]}%
			}{\oldendlemma}
			\let\oldtheorem\theorem
			\let\oldendtheorem\endtheorem
			\renewenvironment{theorem}[1][]{%
				\def\@currentlabelname{#1}\ifstrempty{#1}{\oldtheorem}{\oldtheorem[#1]}%
			}{\oldendtheorem}
		\makeatother

		\foreach \x in {corollary, definition, example, fact, lemma, proposition, remark, theorem} {
			\RegisterTheoremName{\x}{\expandafter\euppercase\x}
		}
	\else
		
		\newenvironment{corollary}{\begin{cor}}{\end{cor}}
		\newenvironment{definition}{\begin{defin}}{\end{defin}}
		\newenvironment{example}{\begin{es}}{\end{es}}
		\newenvironment{lemma}{\begin{lem}}{\end{lem}}
		\newenvironment{proposition}{\begin{prop}}{\end{prop}}
		\newenvironment{remark}{\begin{rem}}{\end{rem}}
		\newenvironment{theorem}{\begin{thm}}{\end{thm}}
	\fi

	\setlistlabel[enumerate,1]{{\rm(\roman*)}}
	\setlistlabel[enumerateq,1]{{\rm(\roman*)}}
	\renewcommand\proofitemizelabeli{}
	\AtBeginEnvironment{proof}{%
		\let\itemize\proofitemize
		\let\enditemize\endproofitemize
	}
	\ifsiam
		\setlist*[proofitemize,1]{itemindent=\parindent}
	\fi
	
	\setseparator{:}
	\newtcolorbox{mybox}[1][]{%
		left=0pt,
		right=0pt,
		top=0pt,
		bottom=0pt,
		colback=MidnightBlue!10,
		colframe=MidnightBlue!20,
		width=\dimexpr\textwidth\relax,
		enlarge left by=0mm,
		boxsep=3pt,
		arc=5pt,outer arc=5pt,
		#1
	}

\makeatletter
	\let\a\relax
	
	\DeclareRobustCommand{\a}{a\@ifstar{$_{\bregman*}^{}$}{$_{\bregman}$}}
	\DeclareRobustCommand{\A}{A\@ifstar{$_{\bregman*}^{}$}{$_{\bregman}$}}
	\DeclareRobustCommand{\B}{B\@ifstar{$_{\bregman*}^{}$}{$_{\bregman}$}}

	\newcommand{\kernel}{\phi}
	\newcommand{\bregman}{\@ifstar{\kernel^*}{\kernel}}

	\newcommand{\D}{\@ifstar\@@D\@D}
	\newcommand{\@D}{\@ifnextchar_{\@D@sub}{\operatorname{D}_{\bregman}}}
	\newcommand{\@@D}{\@ifnextchar_{\@@D@sub}{\operatorname{D}_{\bregman*}}}
	\def\@D@sub_#1{\operatorname{D}_{#1}}
	\def\@@D@sub_#1{\operatorname{D}_{#1^*}}

	\newcommand{\env}{\@ifstar\@@env\@env}
	\newcommand{\@env}{\overleftarrow{\operatorname{env}}{}^{\bregman}}
	\newcommand{\@@env}{\overrightarrow{\operatorname{env}}{}^{\bregman}}

	\newcommand{\klee}{\@ifstar\@@klee\@klee}
	\newcommand{\@klee}{\overleftarrow{\operatorname{klee}}{}^{\bregman}}
	\newcommand{\@@klee}{\overrightarrow{\operatorname{klee}}{}^{\bregman}}

	\newcommand{\coupling}{\Phi}

	\let\OLDpartial\partial
	\renewcommand{\partial}{
		\OLDpartial
		\@latex@warning{`\noexpand\partial\space': maybe `\noexpand\subdiff' was meant? Use `\noexpand\OLDpartial' to suppress this warning}
	}
	\newcommand{\subdiff}{\@ifstar\@@partial\@partial}
	\newcommand{\@partial}{\OLDpartial_{\coupling}}
	\newcommand{\@@partial}{\OLDpartial_{\coupling}}
	
	\renewcommand{\conj}{\@ifstar\@@conj\@conj}
	\newcommand{\@conj}[1]{#1^{\coupling}}
	\newcommand{\@@conj}[1]{#1^{\coupling}}

	\renewcommand{\biconj}{\@ifstar\@@biconj\@biconj}
	\newcommand{\@biconj}[1]{#1^{\coupling\coupling}}
	\newcommand{\@@biconj}[1]{#1^{\coupling\coupling}}
\makeatother

	\newcommand{\infconv}{\mathbin{\square}}
	\newcommand{\supconv}{\mathbin{\Diamond}}

	\DeclareMathOperator*{\epsargmin}{{\varepsilon}\text{-}arg\,min}

	\DeclareMathOperator{\con}{con}
	\DeclareMathOperator{\ran}{rge}

	\renewcommand{\R}{{%
		\mathchoice
		{\mathrm{I\hspace*{-0.22em}R}}
		{\mathrm{I\hspace*{-0.22em}R}}
		{\mathrm{I\hspace*{-0.19em}R}}
		{\mathrm{I\hspace*{-0.17em}R}}
	}}
	\newcommand{\bR}{{\R}}
	\newcommand{\exR}{{\Rinf}}

	\usepackage{xparse}
	\newif\ifshowcomments\showcommentstrue
	\makeatletter
		\newcommand{\disablecolorlinks}{\def\HyColor@UseColor##1{}}
		\newcommand{\defineInlineComment}[2][]{
			\expandafter\gdef\csname #2\endcsname##1{\ifshowcomments{\color{Blue}\footnotesize{\sc\hl[#1]{#2}: }{\it ##1}}\fi}
		}%
	\makeatother

	\newif\ifshowold\showoldtrue
	\newif\ifshownew\shownewtrue
	\newcounter{saveTheorem}
	\newcounter{saveEquation}
	\colorlet{newcolor}{orange!70!red}
	\colorlet{oldcolor}{black!30}
	\ifsiam\else
		\usepackage{aliascnt}
		\newaliascnt{theorem}{dummythm}
	\fi
	\newcommand{\old}[1]{{%
		\disablecolorlinks
		\ifshowold
			{%
				\setcounter{saveEquation}{\value{equation}}%
				\setcounter{saveTheorem}{\value{theorem}}%
				\renewcommand{\thetheorem}{\thesection.\arabic{theorem}\textsuperscript{*}}%
				\renewcommand{\theequation}{\thesection.\arabic{equation}\textsuperscript{*}}%
				{\color{oldcolor}{}#1}%
				\setcounter{equation}{\value{saveEquation}}%
				\setcounter{theorem}{\value{saveTheorem}}%
			}%
		\fi
	}}
	\newcommand{\new}[1]{{%
		\disablecolorlinks
		\ifshownew
			{\color{newcolor}{}#1}%
		\fi
	}}
	\DeclareExpandableDocumentCommand{\change}{O{}m}{%
		{%
			\old{#1}%
		}\new{#2}%
	}

	\let\Par\par

	\defineInlineComment[orange]{manu}
	\defineInlineComment[yellow]{andreas}
	\defineInlineComment[red]{panos}

	\usepackage{pifont}

\usetikzlibrary{spy}

\newlist{assumenum}{enumerate}{1} 
\setlist[assumenum]{leftmargin=2.1cm,label=\textup{(A\arabic*)},font=\bfseries}
\creflabelformat{assumenumi}{#2#1#3}
\Crefname{assumenumi}{Assumption}{Assumptions}

\newcommand{\TheKeywords}{%
	Bregman distance\Sep
	generalized conjugacy\Sep
	duality\Sep
	\(\Phi\)-convexity.%
}
\newcommand{\TheAMSsubj}{
	26B25\Sep 
	49J52\Sep 
	49J53\Sep 
	49M29\Sep
	49N15\Sep
	90C46.
}
\newcommand{\TheFunding}{%
	The work of the first and third author was supported by:
	FWO projects: No. G0A0920N, No. G086318N, No. G086518N;
	Fonds de la Recherche Scientifique -- FNRS, the Fonds Wetenschappelijk Onderzoek--Vlaanderen under EOS Project No. G0F6718N (SeLMA);
	Research Council KU Leuven C1 project No. C14/18/068;
	European Union's Horizon 2020 research and innovation programme under the Marie Sk\l odowska-Curie grant agreement No. 953348;
	and Ford--KU Leuven Research Alliance project No. KUL0075.
	A. Themelis is supported by the Japan Society for the Promotion of Science (JSPS) KAKENHI grant JP21K17710.
}

\author{%
	Emanuel Laude\thanks{%
		KU Leuven,
		Department of Electrical Engineering (ESAT-STADIUS),
		Kasteelpark Arenberg 10, 3001 Leuven, Belgium~
		{\tt%
			\href{mailto:emanuel.laude@esat.kuleuven.be}{\{emanuel.laude,}%
			\href{mailto:panos.patrinos@esat.kuleuven.be}{panos.patrinos\}}%
			\href{mailto:emanuel.laude@esat.kuleuven.be,panos.patrinos@esat.kuleuven.be}{@esat.kuleuven.be}%
		}%
	}%
	\and
	Andreas Themelis\thanks{%
		Kyushu University,
		Faculty of Information Science and Electrical Engineering (ISEE),
		744 Motooka, Nishi-ku, 819-0395 Fukuoka, Japan~
		{\tt
			\href{mailto:andreas.themelis@ees.kyushu-u.ac.jp}{andreas.themelis@ees.kyushu-u.ac.jp}%
		}%
	}%
	\and
	Panagiotis Patrinos\footnotemark[2]%
}%

\ifsiam
	\headers{Dualities for non-Euclidean smoothness and strong convexity}{E. Laude, A. Themelis, and P. Patrinos}
	\title{%
 		Dualities for non-Euclidean smoothness\texorpdfstring{\\}{ }and strong convexity\texorpdfstring{\\}{ }under the light of generalized conjugacy%
		\thanks{%
			Submitted to the editors \today.%
			\funding{\TheFunding}%
		}%
	}
\else
	\title{%
		Dualities for non-Euclidean smoothness\texorpdfstring{\\}{ }%
		and strong convexity\texorpdfstring{\\}{ }%
		under the light of generalized conjugacy%
		\thanks{\TheFunding}%
	}
	\date{}
\fi

\showcommentsfalse         

\begin{document}
	\maketitle
	\begin{abstract}
		Relative smoothness and strong convexity have recently gained considerable attention in optimization. These notions are generalizations of the classical Euclidean notions of smoothness and strong convexity that are known to be dual to each other.
However, conjugate dualities for non-Euclidean relative smoothness and strong convexity remain an open problem as noted earlier by
Lu, Freund and Nesterov [SIAM J. Optim., 28 (2018), pp. 333--354].
In this paper we address this question by introducing the notions of anisotropic strong convexity and smoothness as the respective dual counterparts. The dualities are developed under the light of generalized conjugacy which leads us embed the anticipated dual notions within the superclasses of certain upper and lower envelopes. In contrast to the Euclidean case these inclusions are proper in general as showcased by means of counterexamples.
	\end{abstract}
	

\tableofcontents

	\section{Introduction} \label{sec:intro}
		\subsection{Motivation}
An important duality result in optimization is that a convex function is differentiable with Lipschitz continuous (in the Euclidean sense) gradient mapping if and only if its convex conjugate is strongly convex, i.e., it remains convex after the subtraction of a positive multiple of the squared Euclidean norm, see \cref{fig:euclidean} upper row.
The notions of Euclidean smoothness and strong convexity can be generalized to \emph{relative smoothness} \cite{birnbaum2011distributed,bauschke2017descent,lu2018relatively} and \emph{relative strong convexity} \cite{lu2018relatively,bauschke2019linear} which have recently received considerable attention in the optimization literature \cite{birnbaum2011distributed,bauschke2017descent,lu2018relatively,davis2018stochastic,bauschke2019linear,ahookhosh2021bregman,Mukkamala:2022aa,BSTV2018,dragomir2022optimal}. However, it remains an open problem whether the known conjugate duality between Euclidean Lipschitz smoothness and Euclidean strong convexity can be generalized to the non-Euclidean case as noted earlier in \cite[Section 3.4]{lu2018relatively}.
\begin{figure}[h!]
		\input{TeX/Text/Diagram_Euclidean.tex}
		\caption{%
			Euclidean conjugate dualities for proper, lsc, convex functions $f$ with $L>0$.
			Here,
			$h\infconv \frac{L}{2}\|\cdot\|^2\coloneqq\protect\inf*_{y\in\protect\R^n}h(y)+\protect\frac{L}{2}\|{}\cdot{}-y\|^2$
			and
			$g\supconv \protect\frac{1}{2L}\|\cdot\|^2\coloneqq\sup_{y\in\protect\R^n}\protect\frac{1}{2L}\|{}\cdot{}-y\|^2+g(y)$ 
			denote infimal and supremal convolutions of $h:\bR^n \to \exR$ resp. $g:\bR^n \to \exR$ and a positive multiple of the squared Euclidean norm, respectively.
		}%
		\label{fig:euclidean}%
\end{figure}
\begin{figure}[t!]
	\begin{subfigure}{\linewidth}%
		\input{TeX/Text/Diagram_Bsmooth.tex}%
		\caption{%
			Conjugate duality correspondences between \B-smoothness and \a*-strong convexity as well as Bregman--Moreau envelopes and supremal convolutions for proper, lsc, convex functions \(\func{f}{\R^n}{\exR}\). The implications are equivalences if, in addition, $f$ is essentially strictly convex or, equivalently, $f^*$ is essentially smooth, cf. \cref{thm:duality_aniso_str_cvx}.
		}%
	\end{subfigure}
	\begin{subfigure}{\linewidth}%
		\input{TeX/Text/Diagram_asmooth.tex}
		\caption{%
			Conjugate duality correspondences between \B*-strong convexity and \a-smoothness as well as Bregman--Klee envelopes and infimal convolutions for proper, lsc, convex functions \(\func{f}{\R^n}{\exR}\).%
		}%
	\end{subfigure}
	\caption{%
		Schematics of the interrelations between smoothness, strong convexity, and certain pointwise min/max representations for proper, lsc, convex functions.  
		\B- and \B*- are short for relative to \(\phi\) as in \cref{assum:a1,assum:a2,assum:a3} and to its conjugate \(\phi^*\), respectively, in the classical Bregman sense, see \cref{defin:B-}.
		Similarly, \a- and \a*- refer to the \emph{anisotropic} counterparts defined in \protect\cref{defin:a-}.
		Here, \(\protect\env* g:=\protect\inf_{x \in \bR^n} \protect\D(x, \cdot) + g(x)\) and \(\protect\klee* g\protect:=\protect\sup_{x \in \bR^n} \protect\D(x, \cdot) - g(x)\) are (right) Bregman--Moreau and Bregman--Klee envelopes of \(g:\bR^n \to \exR\), respectively and $\protect\D(x, y)\coloneqq \phi(x) - \phi(y) - \protect\innprod{\protect\nabla\phi(y)}{x-y}$ is the Bregman distance generated by $\phi$.%
	}%
	\label{fig:contributions}%
\end{figure}
The goal of this paper is to close this gap by developing a full\footnote{Here, ``full'' means that we do not exclude nonsmoothness.} conjugate duality for non-Euclidean smoothness and strong convexity.
	Following \cite{penot1990strongly} we examine the duality under the light of $\Phi$-convexity \cite{moreau1970inf} which typically appears in the context of eliminating duality gaps \cite{rockafellar1974augmented,penot1990strongly,bauermeister2021lifting,bui2021zero} in nonconvex and nonsmooth optimization and optimal transport theory, see, e.g., \cite{Vil08}.
	We shall see that the sought duality relation is not found between relative smoothness and relative strong convexity.
	Instead, it involves different notions of relative smoothness and strong convexity which are developed in this paper.
	
	Overall, our contribution is twofold:
	\begin{enumerate}
	\item We furnish the dual counterparts of relative smoothness and relative strong convexity in terms of lower and upper (sub)gradient inequalities that we respectively refer to as \emph{anisotropic strong convexity} and \emph{anisotropic smoothness}.
	The latter is closely related to anisotropic prox-regularity introduced in \cite[Definition 2.13]{laude2021lower}.
		Restricting to convex functions we prove conjugate dualities between relative smoothness and anisotropic strong convexity as well as relative strong convexity and anisotropic smoothness.%
	\item 
		Following \cite{penot1990strongly} the duality correspondences are developed under the light of generalized conjugacy.
		This leads us to study certain superclasses of the anticipated non-Euclidean notions of smoothness and strong convexity examining the properness of these inclusions. In the primal these superclasses correspond to the Bregman--Moreau \cite{bauschke2017regularizing,bauschke2006joint,bauschke2009bregman,kan2012moreau,chen2012moreau,laude2020bregman} and Bregman--Klee envelopes \cite{bauschke2011chebyshev,bauschke2009bregmanklee} and in the dual to infimal and supremal convolutions aka Klee envelopes \cite{cabot2017envelopes,jourani2016nsluc,cabot2018attainment}.
		The known equivalence between the notion of Euclidean weak convexity (aka paraconvexity or hypoconvexity) and the representation in terms of a negative Moreau envelope, see, e.g., \cite{wang2010chebyshev} also motivates the study of non-Euclidean generalizations of the class of weakly convex functions.
	\end{enumerate}
	In the convex case the aforementioned equivalences and implications are synopsized in \cref{fig:contributions}. Further restricting to the Euclidean setting the relations simplify into the ones depicted in \cref{fig:euclidean}.
\subsection{Paper organization}
The remainder of the paper is organized as follows.
\Cref{sec:convexity} offers an overview of basic notions of generalized conjugacy and lists some related results which will be used in the sequel.
\Cref{sec:BregmanAniso} will then revise the classes of relatively smooth and relatively strongly convex functions under the lens of \(\Phi\)-convexity, and introduce their \emph{anisotropic} counterparts which in the subsequent \cref{sec:duality} will be shown to be the sought dual classes.
The shortcomings of \(\Phi\)-convexity in the duality picture will be ultimately showcased in \cref{sec:gap} by means of counterexamples.
\Cref{sec:conclusion} concludes the paper.

\subsection{Preliminaries and notation}
With \(\R\) and \(\Rinf\coloneqq\R\cup\set{\infty,-\infty}\) we respectively denote the real and extended real line.
We denote by $\innprod{\cdot}{\cdot}$ the standard Euclidean inner product on $\bR^n$ and by $\|x\|:=\sqrt{\langle x, x \rangle}$ for any $x\in \bR^n$ the standard Euclidean norm on $\bR^n$.
The closed ball of radius \(r>0\) centered at \(x\in\R^n\) is denoted as \(\cball xr\).
For a set \(S\subseteq\R^n\), \(\interior S\subseteq S\) denotes its interior, \(\con S\supseteq S\) its convex hull, and \(\func{\indicator_S}{\R^n}{\Rinf}\) its indicator function, namely \(\indicator_S(x)=0\) if \(x\in S\) and \(\infty\) otherwise.

The notation \(\ffunc{T}{\R^n}{\R^n}\) indicates a set-valued mapping, whose domain and range are defined as
\(
	\dom T=\set{x\in\R^n}[T(x)\neq\emptyset]
\)
and
\(
	\range T=\bigcup_{x\in\R^n}T(x)
\),
respectively.
We say that \(T\) is locally bounded at \(\bar x\) if there exists a neighborhood \(\mathcal N_{\bar x}\) of \(\bar x\) such that \(\bigcup_{x\in\mathcal N_{\bar x}}T(x)\) is bounded.

The effective domain of an extended real-valued function \(\func{f}{\R^n}{\Rinf}\) is denoted by \(\dom f\coloneqq\set{x\in\R^n}[f(x)<\infty]\), and we say that \(f\) is: proper if $f(x) > -\infty$ for all $x \in \bR^n$ and \(\dom f\neq\emptyset\); lower semicontinuous (lsc) if \(f(\bar x)\leq\liminf_{x\to\bar x}f(x)\) for all \(\bar x\in\R^n\); super-coercive if \(f(x)/\|x\|\to\infty\) as \(\|x\|\to\infty\).
The convex conjugate of \(f\) is denoted as \(f^*\coloneqq\sup_{y \in \bR^n}\set{\innprod{\cdot}{y}-f(y)}\).
We say that \(f\) is strictly continuous at a point \(\bar x\) where $f$ is finite if
\(
	\limsup_{
		\limsubstack{x,x'&\to&\bar x\\x&\neq&x'}
	}{
		\frac{|f(x)-f(x')|}{\|x-x'\|}
	}
{}<{}
	\infty
\).
With \(\mathcal C^k(\R^n)\) we indicate the set of functions from $\bR^n$ to $\bR$ which are \(k\) times continuously differentiable, while with \(\Gamma_0(\R^n)\) the extended real-valued ones which are proper, lsc, and convex.
The identity function \(x\mapsto x\) on \(\R^n\) is denoted by \(\id\).

The set-valued mappings \(\ffunc{\widehat\OLDpartial f,\OLDpartial f}{\R^n}{\R^n}\) are the regular and the limiting subdifferential of \(f\), where \(\bar v\in\widehat\OLDpartial f(\bar x)\) if
\(
	\liminf_{\limsubstack{x&\to&\bar x\\x&\neq&\bar x}}{
		\frac{
			f(x)-f(\bar x)-\innprod{\bar v}{x-\bar x}
		}{
			\|x-\bar x\|
		}
	}
{}\geq{}
	0
\),
while \(\bar v\in\OLDpartial f(\bar x)\) if \(\bar x\in\dom f\) and there exists a sequence \(\graph\widehat\OLDpartial f\ni(x^k,v^k)\to(\bar x,\bar v)\) such that \(f(x^k)\to f(\bar x)\).
If \(g\in\mathcal C^1(\R^n)\), then \(\widehat\OLDpartial(f+g)=\widehat\OLDpartial f+\nabla g\) and \(\OLDpartial(f+g)=\OLDpartial f+\nabla g\) \cite[Exercise 8.8(c)]{RoWe98}.
The inclusion \(0\in\widehat\OLDpartial f(\bar x)\) is necessary for local minimality of \(f\) at \(\bar x\) \cite[Theorem 10.1]{RoWe98}.
We define $\epsargmin f:=\{x : f(x) \leq \inf f + \varepsilon \}$.
We adopt the notions of essential smoothness, essential strict convexity and Legendre type functions from \cite[Section 26]{Roc70}. We say that a function $f \in \Gamma_0(\bR^n)$ is: \emph{essentially smooth} if $\interior(\dom f) \neq \emptyset$, $f$ is differentiable on $\interior(\dom f)$ and $\|\nabla f(x^\nu)\|\to \infty$ whenever $\interior(\dom f) \ni x^\nu \to x \in \boundary \dom f$; \emph{essentially strictly convex} if $f$ is strictly convex on every convex subset of $\dom \partial f$; \emph{Legendre} if $f$ is both essentially smooth and essentially strictly convex.


	\section{Generalized conjugacy}\label{sec:convexity}%
		In this section we recapitulate the notions of $\Phi$-convexity and $\Phi$-conjugacy \cite{moreau1970inf,balder1977extension,dolecki1978convexity,RoWe98,villani2021topics}, which will prove valuable tools in the sequel. 
Throughout this section, $X$ and $Y$ are nonempty sets and $\func{\Phi}{X \times Y}{\bR}$ is a real-valued function called pairing (or coupling).
Although this setting is general enough for our purposes, we emphasize that all the results are true even for extended real-valued couplings, up to the adoption of suitable extended arithmetics to account for indeterminate forms, see, e.g. \cite{moreau1966fonctionnelles,RoWe98}.

\begin{definition}[$\Phi$-convexity]\label{defin:Phicvx}%
	We say that $\func{f}{X}{\exR}$ is (left) $\Phi$-convex if there is $\set*{(y_i, \beta_i) \in Y \times \exR}[i \in \mathcal{I}]$ for some index set $\mathcal{I}$ such that
	\begin{align}
	\label{eq:leftconvex}
		f(x)
	{}={} &
		\sup_{i \in \mathcal{I}} \Phi(x, y_i) - \beta_i
	\quad
		\forall x \in X.
	\shortintertext{%
		When $\mathcal{I}=\emptyset$ we define $f\equiv -\infty$.
		Likewise we say that $\func{g}{Y}{\exR}$ is (right) $\Phi$-convex if there is $\set*{ (x_j, \alpha_j)\in X \times \exR}[ j \in \mathcal{J}]$ for some index set $\mathcal{J}$ such that for all $y \in Y$
	}
	\label{eq:rightconvex}
		g(y)
	{}={} &
		\smash{\sup_{j \in \mathcal{J}}{}} \Phi(x_j, y) - \alpha_j
	\quad
		\forall y \in Y.
	\end{align}
	When $\mathcal{J}=\emptyset$ we define $g\equiv -\infty$.
\end{definition}
Sometimes $\Phi$-convex functions are also referred to as $\Phi$-envelopes. Note that if $X=Y=\bR^n$ and $\Phi=\innprod{\cdot}{\cdot}$ is the Euclidean inner product, one recovers the class of proper convex lsc functions.
This condition, however, differs from the classical secant line definition of convexity which also includes functions which are not lsc.
More generally, a typical choice for the coupling is the negative squared Euclidean norm $\Phi(x,y)=-\frac{1}{2\lambda}\|x-y\|^2$ for $\lambda >0$ which leads to the proximal hull \cite[Example 1.44]{RoWe98}. Another interesting choice is $\Phi(x, y):=\langle v, x\rangle -\frac{r}{2}\|x\|^2$ for $y=(v,r)$. Under this coupling all proper lsc and prox-bounded functions are $\Phi$-convex \cite[Example 11.66]{RoWe98}. Further examples can be found in \cite{balder1977extension,dolecki1978convexity,villani2021topics}.

\begin{definition}[$\Phi$-conjugate functions] \label{def:phi_convex}
	The (left) $\Phi$-conjugate of $\func{f}{X}{\exR}$ on $Y$ is $\func{\conj f}{Y}{\exR}$, defined by
	\begin{align}
	\label{eq:conj}
		\conj f(y)
	{}\coloneqq{} &
		\sup_{x \in X} \Phi(x, y) - f(x),
	\shortintertext{and the (left) $\Phi$-biconjugate of $f$ back on $X$ is $\func{\biconj f}{X}{\exR}$, defined by}
		\biconj f(x)
	{}\coloneqq{} &
		\sup_{y \in Y} \Phi(x, y) - \conj f(y).
	\shortintertext{Likewise, the (right) $\Phi$-conjugate of $\func{g}{Y}{\exR}$ on $X$ is defined by}
	\label{eq:conj*}
		\conj*g(x)
	{}\coloneqq{} &
		\sup_{y \in Y} \Phi(x, y) - g(y),
	\shortintertext{and the (right) $\Phi$-biconjugate of $g$ back on $Y$ is given by}
		\biconj* g(y)
	{}\coloneqq{} &
		\sup_{x \in X} \Phi(x, y) - \conj*g(x).
	\end{align}
\end{definition}
Once again specializing the definition to $X=Y=\bR^n$ and $\Phi=\innprod{\cdot}{\cdot}$ one recovers the familiar convex conjugate and biconjugate operations. If $\Phi(x,y)=-\frac{1}{2\lambda}\|x-y\|^2$ for $\lambda >0$ is the negative squared Euclidean norm $\Phi$-conjugacy yields the proximal transform \cite[Example 11.64]{RoWe98}. Another choice for $\Phi$ is $-\|x-y\|$ which leads to the negative Pasch--Hausdorff envelopes; see, e.g., \cite[Example 9.11]{RoWe98}. The corresponding class of $\Phi$-conjugates is the class of globally Lipschitz functions and plays an important role in the Kantorovich--Rubinstein duality in optimal transport theory; see, e.g., \cite[Remark 6.5]{Vil08}.
Convex subgradients can similarly be captured as a special case of $\Phi$-subgradients, defined next; see, e.g., \cite[Remark 3.3]{balder1977extension}.

\begin{definition}[$\Phi$-subgradients and $\Phi$-subdifferential]
	We say that $\bar y \in Y$ is a (right) $\Phi$-subgradient of $\func{f}{X}{\exR}$ at $\bar x\in\dom f$, denoted by $\bar{y}\in\subdiff f(\bar x)$, if the following generalized subgradient inequality holds:
	\begin{equation}\label{eq:subdiff}
		f(x) \geq f(\bar x) + \Phi(x, \bar y) - \Phi(\bar x, \bar y) \quad \forall x \in X,
	\end{equation}
	or, equivalently, if \(\bar x\in\argmax_{x\in X}\set{\Phi(x,\bar y)-f(x)}\). We call the set $\subdiff f(\bar x)$ of all subgradients $\bar y$ of $f$ at $\bar x$ the (right) $\Phi$-sub\-dif\-fe\-ren\-tial of $f$ at $\bar x$.

	For $\func{g}{Y}{\exR}$ the definition of a (left) $\Phi$-subgradient and (left) $\Phi$-sub\-dif\-fe\-ren\-tial is parallel.
\end{definition}

The remainder of this section collects a list of useful properties of generalized convexity which will be invoked later on in the manuscript.
Although quite standard, see e.g., \cite{balder1977extension,dolecki1978convexity,RoWe98,villani2021topics}, these results are instrumental in our analysis, and for completeness we will thus detail the simple proofs.
We confine the discussion to ``left'' functions \(\func fX\Rinf\); the results for ``right'' functions \(\func gY\Rinf\) are parallel.

The following result can be regarded as a generalization of the \emph{Fenchel--Moreau theorem}, see, e.g., \cite[Theorem 2.78]{mordukhovich2013easy}.
\begin{lemma}
	For any $\func{f}{X}{\exR}$, the following properties hold:
	\begin{enumerate}
	\item\label{thm:conj:conv}%
		$\func{\conj f}{Y}{\Rinf}$ is right \(\Phi\)-convex.
	\item\label{thm:bconj:conv}%
		$\func{\biconj f}{X}{\exR}$ is the 
		largest left $\Phi$-convex function
		majorized by $f$; namely, for
		$$
		\mathcal{A}_f:=\{(y,\beta) \in Y \times \exR : \Phi(x,y) -\beta \leq f(x)~~\forall x \in X\},
		$$
		we have
		\[
			f(\bar x)
		{}\geq{}
			\biconj f(\bar x)
		{}={}
			\sup_{(y,\beta) \in \mathcal{A}_f} \{\Phi(\bar x,y) -\beta\}
		\quad
			\forall \bar x\in X,
		\]
		where the inequality holds with equality if, in addition, $f$ is $\Phi$-convex.
	\item\label{thm:bbconj}%
		$\conj{(\biconj f)}=\conj f$.
	\item\label{thm:FYineq}%
		\emph{(\(\Phi\)-Fenchel--Young inequality)}
		$f(x) + \conj f(y) \geq \Phi(x, y)$ for every $(x,y)\in X\times Y$.
	\end{enumerate}
\end{lemma}
\begin{proof}~
	\begin{itemize}
	\item``\ref{thm:conj:conv}''
		Trivial (compare \eqref{eq:conj} and \eqref{eq:rightconvex}).
	\item``\ref{thm:bconj:conv}''
		\(\biconj f\) is left \(\Phi\)-convex by assertion \ref{thm:conj:conv}, being it the right \(\Phi\)-conjugate of \(\conj f\).
		Moreover, for \(\bar x\in X\) one has
		\begin{align*}
			\biconj f(\bar x)
		{}\defeq{} &
			\sup_{y \in Y}\set{
				\Phi(\bar x,y)-\conj f(y)
			}
		{}={}
			\sup_{(y,\beta)\in Y \times \exR}\set{
				\Phi(\bar x,y)-\beta
			}[
				\conj f(y)
			{}\leq{}
				\beta
			]
		\\
		{}={} &
			\sup_{(y,\beta) \in \mathcal{A}_f} \{\Phi(\bar x,y)-\beta\},
		\end{align*}
		where the last identity holds since $\conj f(y)\defeq \sup_{x \in X} \Phi(x, y) - f(x) \leq \beta$ if and only if $(y, \beta) \in \mathcal{A}_f$. In particular this implies the inequality $f(\bar x) \geq \biconj f(\bar x)$. If, in addition, $f$ is $\Phi$-convex there is $\{ (y_i, \beta_i) \in Y \times \exR : i \in \mathcal{I} \}$ for some index set $\mathcal{I}$ such that $f = \sup_{i \in \mathcal{I}} \Phi(\cdot, y_i) - \beta_i$. Therefore for any $i \in \mathcal{I}$ we have $f(x)\geq \Phi(x, y_i) - \beta_i$ for all $x \in X$ and thus $\mathcal{I}\subseteq \mathcal{A}_f$. This implies that $f(\bar x)\leq \sup_{(y,\beta) \in \mathcal{A}_f} \{\Phi(\bar x,y)-\beta\} = \biconj f(\bar x)$ and thus $f(\bar x) =\biconj f(\bar x)$ as claimed.
	\item``\ref{thm:bbconj}''
		By definition $\biconj f = \conj{(\conj f)}$. Thus we have \(\biconj{(\conj f)} =\conj{(\conj{(\conj f)})}=\conj{(\biconj f)}\) and the claim follows from assertions \ref{thm:conj:conv} and \ref{thm:bconj:conv}, since \(\conj f\) is right \(\Phi\)-convex and as such it coincides with its biconjugate.
	\item``\ref{thm:FYineq}''
		Follows from the definition of left \(\Phi\)-conjugate, cf. \eqref{eq:conj}.
	\qedhere
	\end{itemize}
\end{proof}

\begin{lemma}\label{thm:Phicvx}%
	For any $\func{f}{X}{\exR}$, the following are equivalent:
	\begin{enumerateq}
	\item\label{thm:phiconvex:}%
		$f$ is left $\Phi$-convex;
	\item\label{thm:phiconvex:f=biconj}%
		$\biconj f=f$;
	\item\label{thm:phiconvex:conj*}%
		$f=\conj*g$ for some $\func{g}{Y}{\exR}$.
	\end{enumerateq}
\end{lemma}
\begin{proof}~
	\begin{itemize}
	\item``\ref{thm:phiconvex:} $\Rightarrow$ \ref{thm:phiconvex:f=biconj}''
		Follows from \cref{thm:bconj:conv}.
	\item``\ref{thm:phiconvex:f=biconj} $\Rightarrow$ \ref{thm:phiconvex:conj*}''
		Holds with $g=\conj f$.
	\item``\ref{thm:phiconvex:conj*} $\Rightarrow$ \ref{thm:phiconvex:}''
		Trivial, since $\conj*g$ is left $\Phi$-convex by \cref{thm:conj:conv}.
	\qedhere
	\end{itemize}
\end{proof}

\begin{lemma}[Fenchel \(\Phi\)-duality]\label{thm:phi_subgrad_existence}%
	Let $\func{f}{X}{\exR}$ be proper.
	For any $\bar x\in X$ and $\bar y\in Y$, the following statements are equivalent:
	\begin{enumerateq}
	\item\label{thm:FY:subdiff}%
		$\bar y \in \subdiff f(\bar x)$;
	\item\label{thm:FY:=}%
		$f(\bar x) + \conj f(\bar y) = \Phi(\bar x,\bar y) \in\R$;
	\item\label{thm:FY:argmin}%
		$\bar x\in\argmin_{x\in X}\set{f(x)-\Phi(x, \bar y)}$;
	\item\label{thm:FY:subdiff*}%
		$\bar x\in\subdiff*\conj f(\bar y)\cap\dom\subdiff f$.
	\end{enumerateq}
	Any one of these equivalent conditions implies \(f(\bar x)=\biconj f(\bar x)\in\R\)\ and \(\bar y\in\subdiff\biconj f(\bar x)\).
	Moreover, when \(f\) is left \(\Phi\)-convex the intersection with \(\dom\subdiff f\) is superfluous in the last condition.
\end{lemma}
\begin{proof}
	The equivalence of the first three assertions follows from the definitions of \(\Phi\)-con\-ju\-gate and \(\Phi\)-subdifferential, cf. \eqref{eq:conj} and \eqref{eq:subdiff}, and in particular each implies the inclusion \(\bar x\in\dom\subdiff f\). Suppose now that \ref{thm:FY:=} holds. Then we have%
		\[
			f(\bar x) \geq f^{\Phi\Phi}(\bar x) \defeq \sup_{y \in Y} \Phi(\bar x, y) - \conj f(y) \geq \Phi(\bar x, \bar y) - f^\Phi(\bar y) = f(\bar x),
		\]
		where the first inequality holds due to \cref{thm:bconj:conv} and the last equality is assertion \ref{thm:FY:=}.
		All inequalities thus hold as equalities showing that $f(\bar x) = \biconj f(\bar x)$ and $\bar y \in \argmax \Phi(\bar x, \cdot) - \conj f$ which means by definition
		\(\bar x\in\subdiff*\conj f(\bar y)\).
		This concludes the proof of the implications
	``\ref{thm:FY:subdiff}
	\(\Leftrightarrow\)
	\ref{thm:FY:=}
	\(\Leftrightarrow\)
	\ref{thm:FY:argmin}
	\(\Rightarrow\)
	\ref{thm:FY:subdiff*}''.

	Suppose now that assertion \ref{thm:FY:subdiff*} holds.
		Since \(\bar x\in\dom\subdiff f\), there exists \(\eta\in\subdiff f(\bar x)\), which by the equivalence of the first two assertions implies
		\(
			f(\bar x)
		{}={}
			\Phi(\bar x,\eta)-\conj f(\eta)
		{}\in{}
			\R
		\).
		Similarly, the inclusion \(\bar x\in\subdiff*\conj f(\bar y)\) implies, through the same equivalence applied to \(\func{\conj f}{Y}{\Rinf}\), that
		\(
			\conj f(\bar y)+\biconj f(\bar x)
		{}={}
			\Phi(\bar x,\bar y)
		\). Then,
	\[
		f(\bar x)
	{}={}
		\Phi(\bar x,\eta)-\conj f(\eta)
	{}\leq{}
		\sup\set{\Phi(\bar x,{}\cdot{})-\conj f}
	{}\defeq{}
		\biconj f(\bar x)
	{}\overrel*[\leq]{\ref{thm:bconj:conv}}{}
		f(\bar x),
	\]
	Therefore, \(\Phi(\bar x,\bar y)-\conj f(\bar y){}=\biconj f(\bar x)=f(\bar x){}\in\R\), which is assertion \ref{thm:FY:=}.
	
	To conclude, observe that
		\(
			\bar y \in \subdiff f(\bar x)
		{}\Rightarrow{}
			\bar x \in \subdiff*\conj f(\bar y)
		{}\Rightarrow{}
			\bar y \in \subdiff\biconj f(\bar x)
		\),
		owing to the shown equivalence of the first and last assertions, applied first to \(\func fX\Rinf\) and then to \(\func{\conj f}Y\Rinf\).
	If \(f\) is left \(\Phi\)-convex, then \(f=\biconj f\) by \cref{thm:Phicvx}, and the above chain of implications then reduces to
		\(
			\bar y \in \subdiff f(\bar x)
		{}\Leftrightarrow{}
			\bar x \in \subdiff*\conj f(\bar y)
		\),
	without domain inclusion.
\end{proof}

	\section{Bregman and anisotropic convexity and smoothness}\label{sec:BregmanAniso}%
		\subsection{Standing assumptions}
The notions considered in the remainder of this manuscript involve a \emph{reference function} $\phi$ which complies with the following standing requirements, that will be henceforth assumed without further mention:
\begin{mybox}
\begin{assumenum}
\item \label{assum:a1} $\phi\in\Gamma_0(\bR^m)$ is Legendre;
\item \label{assum:a2} $\dom \phi=\bR^n$;
\item \label{assum:a3} $\phi$ is super-coercive.
\end{assumenum}
\end{mybox}
\Cref{assum:a2,assum:a3} are made for simplicity so as to avoid treatment of complicating boundary cases and extended arithmetics.
In view of \cite[Proposition 2.16]{bauschke1997legendre}, the convex conjugate $\phi^*$ of $\phi$ also complies with the same requirements, and the gradients $\func{\nabla\phi,\nabla\phi^*}{\bR^n}{\bR^n}$ are mutually inverse bijections: $(\nabla \phi)^{-1}=\nabla \phi^*$.

\subsection{Bregman convexity and smoothness}
In this section we revisit the existing notions of relative strong/weak convexity \cite{chen2012moreau,lu2018relatively,bauschke2019linear} and relative smoothness \cite{birnbaum2011distributed,bauschke2017descent,lu2018relatively} and collect some basic properties. In addition, we provide alternative characterizations in terms of $\Phi$-convexity. To begin with we introduce the Bregman distance \cite{Bre67} $\D$ generated by the reference function $\phi$ which thanks to the full domain property of $\phi$ is defined as
	\[
		\D(x,y)\coloneqq \phi(x) - \phi(y) - \innprod{\nabla\phi(y)}{x-y}
	\]
	for any $x,y\in\bR^n$.
	Next, we recall the definitions of relative weak convexity \cite{chen2012moreau}, relative strong convexity \cite{lu2018relatively,bauschke2019linear} and a two-sided version \cite{BSTV2018} of relative smoothness \cite{birnbaum2011distributed,bauschke2017descent,lu2018relatively}.
	Since our dual notions involve a reference function as well, we will refine the existing terminology and henceforth refer to relative weak/strong convexity and smoothness as Bregman weak/strong convexity and smoothness (\B-weak/\B-strong convexity and \B-smoothness):
\begin{mybox}
	\begin{definition}[Bregman weak/strong convexity and smoothness]\label{defin:B-}%
		A function $\func{f}{\bR^n}{\exR}$ is said to be
		\begin{enumerate}
		\item\label{def:B-weak}%
			\emph{\B-weakly convex} if \(f+\phi\) is convex;
		\item\label{def:B-strong}%
			\emph{\B-strongly convex} if \(f-\phi\) is convex;
		\item\label{def:B-smooth}%
			\emph{\B-smooth} if both \(f\) and \(-f\) are proper and \B-weakly convex.
		\end{enumerate}
	\end{definition}
\end{mybox}
A source of practical examples for \B-weakly convex functions is the class of compositions $c \circ F : \bR^n \to \bR$ for a convex and globally Lipschitz function $c:\bR^m \to \bR$ and a nonlinear mapping $F:\bR^n \to \bR^m$ whose component functions $F_i:\bR^n \to \bR$ are \B-smooth, see, e.g., \cite{laude2020bregman,Mukkamala:2022aa}.

Next we provide basic equivalent characterizations of \B-weak/strong convexity and \B-smooth\-ness in terms of lower and upper (sub)gradient inequalities. The characterization of \B-smoothness is a specialization of the two-sided \emph{full extended descent lemma} \cite[Lemma 2.1]{BSTV2018} to $\phi$ with full domain; see \cite{bauschke2017descent,lu2018relatively} for the one-sided version. For a joint treatment we introduce a flag $r \in \{\pm 1\}$ such that convexity of $f -r\phi$ means either \B-weak ($r=-1$) or \B-strong ($r=+1$) convexity.

\begin{proposition}\label{thm:B-weak}%
	Let \(\func{f}{\R^n}{\Rinf}\) be proper and lsc, and let \(r=-1\) (resp. \(r=+1\)).
	The following are equivalent:
	\begin{enumerateq}
	\item \label{thm:B-:i} \(f\) is \B-weakly convex (resp. \B-strongly convex), that is, \(f-r\phi\) is convex;
	\item \label{thm:B-:ineq}%
		for every \((\bar x,\bar v)\in\graph\OLDpartial f\), it holds that
		\begin{equation}\label{eq:B-ineq}
			f(x){}\geq{} f(\bar x) {}+{} \innprod{\bar v}{x-\bar x} {}+{}r\D(x,\bar x) \quad \forall x\in\R^n;
		\end{equation}
	\item \label{thm:B-:iii}
		$\innprod{x-x'}{v-v'}\geq r\innprod{x-x'}{\nabla\phi(x) -\nabla\phi(x')}$ holds for all $(x, v), (x',v') \in\graph\OLDpartial f$.
	\end{enumerateq}
	Moreover, \(f\) is \B-smooth if and only if it is continuously differentiable and satisfies%
	\begin{equation}\label{eq:B-smooth}
		| f(x) {}-{} f(\bar x) {}-{} \innprod{\nabla f(\bar x)}{x-\bar x}| {}\leq{} \D(x,\bar x) \quad \forall x\in\R^n.
	\end{equation}
\end{proposition}
\begin{proof}
	We start by observing that, since $\phi$ is smooth, $\OLDpartial (f-r\phi)=\OLDpartial f-r\nabla\phi$.
	\begin{itemize}
	\item``\ref{thm:B-:i} \(\Rightarrow\) \ref{thm:B-:ineq}''
		For \((\bar x,\bar v)\in\graph\OLDpartial f\), convexity of $f{}-{}r\phi$ yields
		$$
		(f{}-{}r\phi)(x) \geq (f{}-{}r\phi)(\bar x) + \innprod{\bar v{}-{}r\nabla \phi(\bar x)}{x-\bar x}\quad \forall x \in \bR^n,
		$$
		which after reordering is \cref{eq:B-ineq}.
	\item``\ref{thm:B-:ineq} \(\Rightarrow\) \ref{thm:B-:iii}'' For $(x, v), (x',v') \in \graph \OLDpartial f$, \cref{eq:B-ineq} yields
		$$
			f(x') \geq f(x) + \innprod{v}{x' - x} {}+{} r\D(x', x)
		~~\text{and}~~
			f(x) \geq f(x') + \innprod{v'}{x - x'} {}+{} r\D(x, x').
		$$
		By summing the two inequalities and reordering, assertion \ref{thm:B-:iii} is obtained.
	\item``\ref{thm:B-:iii} \(\Rightarrow\) \ref{thm:B-:i}''
		By assumption, for all $(x, v), (x',v') \in\graph\OLDpartial f$ it holds that $\innprod{x-x'}{v-v'}\geq r\innprod{x-x'}{\nabla\phi(x) -\nabla\phi(x')}$.
		By suitably rearranging, this can equivalently be written as
	\ifsiam\else
		the inequality
	\fi
		$\innprod{x-x'}{v-r\nabla \phi(x)-(v'-r\nabla \phi(x'))}\geq 0$.
		Let now two pairs $(x,y), (x',y') \in \graph \OLDpartial (f{}-{}r\phi)$ be fixed.
		Since $\OLDpartial (f{}-{}r\phi) = \OLDpartial f {}-{} r\nabla \phi$ this means that $y = v-r\nabla \phi(x)$ for some $v \in \OLDpartial f(x)$ and $y' = v'-r\nabla \phi(x)$ for some $v' \in \OLDpartial f(x')$. In light of the previous inequality we obtain $\innprod{x-x'}{y-y'}\geq 0$. This implies that $\OLDpartial (f{}-{}r\phi)$ is monotone, and convexity of \(f{}-{}r\phi\) then follows from \cite[Theorem 12.17]{RoWe98}.
	
	\item	We now prove the final claim.
		If \(f\) is \B-smooth, since both $f$ and $-f$ are proper, $f$ must be finite-valued. It then follows from \cite[Proposition 2.5]{ahookhosh2021bregman} that $f$ is continuously differentiable.
		Since, by definition, $f$ and $-f$ are both \B-weakly convex, in light of the shown equivalence among assertions \ref{thm:B-:i} and \ref{thm:B-:ineq}, \eqref{eq:B-ineq} holds for both $f$ and $-f$ with $r=-1$.
		By combining the two resulting inequalities and using the fact that $\OLDpartial f=\nabla f$, \eqref{eq:B-smooth} is obtained.
		Conversely, if \eqref{eq:B-smooth} holds and \(f\) is continuously differentiable, then both $f$ and $-f$ satisfy \eqref{eq:B-ineq} with $r=-1$.
		By invoking again the shown equivalence of the first two assertions, we conclude that both $f$ and $-f$ are (proper and) \B-weakly convex, hence \B-smooth by definition.
		\qedhere
		\end{itemize}%
\end{proof}
Note that in contrast to \cite[Lemma 2.1]{BSTV2018} we do not assume smoothness in \cref{def:B-smooth} as this is implied automatically; also see \cite[Propostion 2.5]{ahookhosh2021bregman}.
Under the restriction that $f$ is finite-valued and locally Lipschitz the equivalence between \cref{thm:B-:i,thm:B-:ineq} in the \B-weakly convex case, i.e., for $r=-1$ was provided in \cite[Lemma 2.2]{davis2018stochastic} for more general Legendre functions $\phi$ that possibly do not have full domain.

In addition, we provide an alternative definition of \B-weak/strong convexity and \B-smoothness involving an \emph{envelope} representation that identifies a function as a hull of supporting elementary functions.
This observation naturally leads to framing these notions within the scope of \(\Phi\)-convexity, cf. \cref{defin:Phicvx}, and in doing so it enables the possibility to employ its general yet very powerful tools.
In this perspective, note that with \(\Phi=-\D\), left and right \(\Phi\)-conjugates amount to left and right (negative) Bregman--Moreau envelopes \(\conj f=-\env f\) and \(\conj*g=-\env*g\), where
\begin{align}
\label{eq:BMenv}
	\env f(y)
{}\coloneqq{}
	\inf\set{f+\D({}\cdot{},y)}
\quad\text{and}\quad &
	\env*g(x)
{}\coloneqq{}
	\inf\set{g+\D(x,{}\cdot{})}.
\shortintertext{%
	Likewise, the attainment sets, which constitute the respective left and right Bregman proximal mappings, are the inverse \(\Phi\)-subdifferentials.
	By changing sign of the coupling function, negative Bregman--Moreau envelopes are replaced by their Bregman--Klee counterparts:
	for \(\Phi=\D\), one has that \(\conj f=\klee f\) and \(\conj*g=\klee*g\), where
}
\label{eq:BKenv}
	\klee f(y)
{}\coloneqq{}
	\sup\set{\D({}\cdot{},y)-f}
\quad\text{and}\quad &
	\klee*g(x)
{}\coloneqq{}
	\sup\set{\D(x,{}\cdot{})-g},
\end{align}
and once again the attainment sets coincide with the inverse \(\Phi\)-subdifferentials.

	First we provide some useful formulas that express these Bregman envelopes in terms of convex conjugates.
	To this end, we introduce the following convenient notation for the reflection:
	\[
		h_-(x)\coloneqq h(-x).
	\]
	We point out that reflection and convex conjugation commute, namely,
	\begin{equation}\label{eq:conjugate_minus}
		(h_-)^*=(h^*)_-
	\end{equation}
	holds for any \(\func{h}{\R^n}{\Rinf}\), see, e.g. \cite[Lemma 7.1]{cabot2017envelopes} or \cite[Proposition 13.23(v)]{BaCo110}.
	This fact allows us to unambiguously omit the brackets in favor of a lighter notation $h_-^*$.

	\begin{lemma}\label{thm:MoreauKlee}%
		For any \(\func{f,g}{\R^n}{\Rinf}\), the following identities hold:
		\begin{enumerate}
		\item\label{thm:env*_conj}%
			\(\env* g=\phi-(g\circ\nabla\phi^*+\phi^*)^*\).
		\item\label{thm:env_conj}%
			\(\env f\circ\nabla\phi^*=\phi^*-(f+\phi )^*\).
		\item\label{thm:env_biconj}%
			\(-\env*(-\env f)=(f + \phi)^{**}-\phi\).
		\item\label{thm:klee*_conj}%
			\(\klee* g=\phi+(g\circ\nabla\phi^*-\phi^*)_-^*\).
		\item\label{thm:klee_conj}%
			\(\klee f\circ\nabla\phi^*=\phi^*+(f-\phi)_-^*\).
		\item\label{thm:klee_biconj}%
			\(\klee*(\klee f)=\phi+(f-\phi)^{**}\).
		\end{enumerate}
	\end{lemma}
\begin{proof}
		The first three identities follow by \cite[Proposition 2.4(ii)]{bauschke2017regularizing}, \cite[Theorem 2.4]{kan2012moreau} and \cite[Proposition 2.14]{wang2021bregman} respectively.
		Thus it suffices to show the identities involving the Klee envelopes.
		Thanks to the fact that $\nabla \phi(\nabla\phi^*(v))=v$ and $\langle v,\nabla \phi^*(v) \rangle=\phi^*(v) + \phi(\nabla \phi^*(v))$, we have that
		\begin{equation}\label{eq:Bdistphi*}
			\D({}\cdot{},\nabla\phi^*(v))
		{}={}
			\phi - \phi(\nabla\phi^*(v))-\innprod{v}{{}\cdot{}-\nabla\phi^*(v)}
		{}={}
			\phi + \phi^*(v)-\innprod{v}{{}\cdot{}}
		\end{equation}
		holds for every \(v\in\R^n\).
		By using the fact that $\func{\nabla \phi^*}{\bR^n}{\bR^n}$ is surjective, for \(x\in\R^n\) we have
		\begin{align*}
			\klee* g(x)
		&=
			\sup_{v \in \bR^n}\set{\D(x,\nabla \phi^*(v))-g(\nabla\phi^*(v))}
		\\
		&\overrel*{\eqref{eq:Bdistphi*}}
			\sup_{v \in \bR^n}\set{\phi(x) + \innprod{-x}{v}-\bigl(g(\nabla\phi^*(v))-\phi^*(v)\bigr)},
		\shortintertext{%
			and assertion \ref{thm:klee*_conj} follows from the definition of the convex conjugate.
			Similarly, for any $v\in \bR^n$
		}
			\klee f\circ\nabla \phi^*(v)
		&=
			\sup_{x \in \bR^n}\set{\D(x,\nabla \phi^*(v))-f(x)}
		\\
		&\overrel*{\eqref{eq:Bdistphi*}}
			\sup_{x \in \bR^n}\set{\phi^*(v)+\innprod{x}{-v}-\bigl(f(x)-\phi(x)\bigr)},
		\end{align*}
		which yields assertion \ref{thm:klee_conj}. 
		As a consequence of \ref{thm:klee_conj}, we have $(f-\phi)_-^*= \klee f\circ\nabla\phi^* - \phi^*$. Taking conjugates on both sides yields $((f-\phi)_-^*)^* = (\klee f\circ\nabla\phi^* - \phi^*)^*$. Thus we have the following equalities
		$$
		((f-\phi)_-^*)_-^* + \phi = (\klee f\circ\nabla\phi^* - \phi^*)_-^* + \phi=  \klee*(\klee f),
		$$
		where the last equality follows from \ref{thm:klee*_conj}. 
		In view of \cref{eq:conjugate_minus} we have $((f-\phi)_-^*)_-^*=(f-\phi)^{**}$ and \ref{thm:klee_biconj} is obtained. 
	\end{proof}

Acquainted with these preliminaries, we can detail the anticipated equivalent formulation of \cref{defin:B-} in terms of \(\Phi\)-convexity. 
\begin{proposition}[\(\Phi\)- and hull characterization of \B-convexity]\label{thm:B-Phi}%
	Let \(r=-1\) (resp. \(r=+1\)) and let \(\Phi=r\D\).
	For any proper \(\func{f}{\R^n}{\Rinf}\) the following are equivalent:%
	\begin{enumerateq}
	\item\label{thm:B-Phi:B-conv}%
		\(f\) is lsc and \B-weakly convex (resp. \B-strongly convex);
	\item\label{thm:B-Phi:Phi-conv}%
		$f$ is $\Phi$-convex;
	\item\label{thm:B-Phi:bienv}%
		$f=-\env*(-\env f)$~ (resp. $f=\klee*(\klee f)$);
	\item\label{thm:B-Phi:env}%
		there exists a function $\func{g}{\bR^n}{\exR}$ such that $f=-\env*g$~ (resp. $f=\klee*g$);
	\item\label{thm:B-Phi:sup}%
		$f$ is the pointwise supremum over all functions of the form $r\phi + \innprod{v}{\cdot} + \delta$, with $(v, \delta) \in \bR^n \times \exR$, that are majorized by $f$;
	\item \label{thm:B-Phi:6}
		\(f\) is lsc and \(\subdiff f= \nabla \phi^* \circ (\nabla\phi-r\OLDpartial f)\).
	\end{enumerateq}
\end{proposition}
\begin{proof}
	\begin{itemize}
	\item``\ref{thm:B-Phi:env} \(\Rightarrow\) \ref{thm:B-Phi:B-conv}''
		By assumption there is a function $g: \bR^n \to \exR$ such that $f=-\env*g$ (resp. $f=\klee*g$). It follows from \cref{thm:env*_conj} (resp. \cref{thm:klee*_conj}) that $f=-\env* g=(g\circ\nabla\phi^*+\phi^*)^* - \phi$ (resp. $f=\klee* g=\phi+(g\circ\nabla\phi^*-\phi^*)_-^*$). Since convex conjugates, as pointwise suprema over affine functions, are convex lsc and $\phi$ is smooth, $f$ is lsc and \B-weakly convex (resp. \B-strongly convex).
	\item``\ref{thm:B-Phi:B-conv} \(\Rightarrow\) \ref{thm:B-Phi:bienv}''
		If \(f-r\phi\) is proper lsc and convex, then \((f-r\phi)^{**}=f-r\phi\), and the claim then follows from \cref{thm:env_biconj} (resp. \cref{thm:klee_biconj}).
	\item``\ref{thm:B-Phi:bienv} \(\Rightarrow\) \ref{thm:B-Phi:env}''
			Consider $g=-\env f$ (resp. $g=\klee f$).
	\item``\ref{thm:B-Phi:Phi-conv} \(\Leftrightarrow\) \ref{thm:B-Phi:bienv}''
			This follows from \cref{thm:Phicvx} via the choice $\Phi:=r\D$, noting that $\biconj f \defeq -\env*(-\env f)$ (resp. $\biconj f \defeq \klee*(\klee f)$).
	\item``\ref{thm:B-Phi:Phi-conv} \(\Leftrightarrow\) \ref{thm:B-Phi:sup}''
			In light of \cref{thm:Phicvx}, \(\Phi\)-convexity is equivalent to the identity \(f=\biconj f\).
			By using the surjectivity of the map \(v\mapsto\nabla\phi^*(-rv)\), we have (recall that \(\Phi=r\D\)) invoking \cref{thm:bconj:conv}%
			\begin{equation}\label{eq:BregBiconj}
				\biconj f
			{}={}
				\sup_{(v,\beta)\in\bR^n\times\exR}\set*{r\D({}\cdot{},\nabla\phi^*(-rv))-\beta}[r\D({}\cdot{},\nabla\phi^*(-rv))-\beta\leq f].
			\end{equation}
			Since
			\(
				r\D({}\cdot{},\nabla\phi^*(-rv))
			{}={}
				r\phi+r\phi^*(-rv)+\innprod{v}{{}\cdot{}}
			\)
			by \eqref{eq:Bdistphi*}, the change of variable \(\delta=r\phi^*(-rv)-\beta\) in \eqref{eq:BregBiconj} results in the supremum over the functions \(r\phi+\innprod{v}{{}\cdot{}}+\delta\) majorized by \(f\), as in assertion \ref{thm:B-Phi:sup}.
	\item ``\ref{thm:B-Phi:B-conv} \(\Rightarrow\) \ref{thm:B-Phi:6}''
			Let $f$ be lsc and \B-weakly convex, and for $\bar x \in \dom (\nabla \phi^* \circ (\nabla \phi -r\OLDpartial f))$ let $\bar y \in \nabla \phi^*(\nabla \phi(\bar x) -r \OLDpartial f(\bar x))$, so that there exists $\bar v \in \OLDpartial f(\bar x)$ such that
			$$
				\bar y = \nabla \phi^*(\nabla \phi(\bar x) -r\bar v).
			$$
			From \cref{thm:B-weak} we obtain (recall that $\Phi:=r\D$)
			\begin{align*}
				f(x) &\geq f(\bar x) + \innprod{\bar v}{x - \bar x} + r\D(x, \bar x) \\
				&{}= f(\bar x) + r\phi(x)-r\phi(\bar x)-r\innprod{\nabla\phi(\bar x)-r\bar v}{x-\bar x} \\
				&\overrel{\eqref{eq:Bdistphi*}}
				f(\bar x)
				{}+{}
				r\D(x,\nabla\phi^*(\nabla\phi(\bar x)-r\bar v))
				{}-{}
				r\D(\bar x, \nabla\phi^*(\nabla\phi(\bar x)-r\bar v)) \\
				&=f(\bar x)+\Phi(x,\bar y)-\Phi(\bar x,\bar y)
				\quad
				\forall x \in \bR^n,
			\end{align*}
			and thus $\bar y \in \subdiff f(\bar x)$ by definition.
			Now choose $\bar y \in \subdiff f(\bar x)$.
			In view of \cref{thm:phi_subgrad_existence}, this means that $\bar x \in \argmin_{x \in \bR^n} f(x) -r\D(x, \bar y)$, and therefore
			\(
				0 \in \OLDpartial f(\bar x) -r\nabla \phi(\bar x) +r\nabla \phi(\bar y)
			\).
			Reordering yields $\bar y \in \nabla \phi^*(\nabla \phi(\bar x)-r\OLDpartial f(\bar x))$.
		\item ``\ref{thm:B-Phi:6} \(\Rightarrow\) \ref{thm:B-Phi:B-conv}''
				Conversely, let \(f\) be lsc and assume that $\subdiff f= \nabla \phi^* \circ (\nabla\phi-r\OLDpartial f)$.
				Let $\bar x \in \dom \OLDpartial f$ and for $\bar v \in \OLDpartial f(\bar x)$ let \(\bar y\coloneqq\nabla \phi^*(\nabla \phi(\bar x) -r\bar v)\).
				Thus we have \(\bar y\in \subdiff f(\bar x)\) which means by definition (again recall that $\Phi:=r\D$)
				$$
					f(x) \geq f(\bar x) +r\D(x, \bar y)-r\D(\bar x, \bar y) \quad \forall x \in \bR^n.
				$$
				Applying the calculation above using \cref{eq:Bdistphi*} we obtain \eqref{eq:B-ineq} and thus in view of \cref{thm:B-weak} $f$ is \B-convex.
	\qedhere
	\end{itemize}
\end{proof}
The result above can be seen as a generalization of \cite[Example 11.26(d)]{RoWe98} for the Euclidean proximal hull.
We remark that the equivalence among \cref{thm:B-Phi:B-conv,thm:B-Phi:bienv,thm:B-Phi:sup} was previously shown in \cite[Proposition 2.14]{wang2021bregman} for the weakly convex case (i.e. $r=-1$) for more general reference functions $\phi$ that need not have full domain.

In light of \cref{thm:B-Phi} we have the following result:
\begin{corollary}[\B-smooth functions as pointwise infima]\label{thm:B-smooth:inf}%
	For any \(f\in\Gamma_0(\R^n)\) the following are equivalent:
	\begin{enumerateq}
	\item\label{thm:B-smooth:inf:smooth}
		\(f\) is \B-smooth;
	\item\label{thm:B-smooth:inf:env}
		\(f=\env* g\) for some proper function $\func{g}{\bR^n}{\exR}$.
	\end{enumerateq}
\end{corollary}
\begin{proof}
	\begin{itemize}
	\item``\ref{thm:B-smooth:inf:smooth} \(\Rightarrow\) \ref{thm:B-smooth:inf:env}''
		Let \(f \in \Gamma_0(\bR^n)\) be \B-smooth. This means that both $-f$ and $f$ are \B-weakly convex and proper. This implies that $f$ is finite-valued and by convexity continuous implying that $-f$ is lsc. In light of \cref{thm:B-Phi} this means that there is some $g:\bR^n \to \exR$ such that $-f = -\env* g$ implying that $f = \env* g$.
	\item``\ref{thm:B-smooth:inf:env} \(\Rightarrow\) \ref{thm:B-smooth:inf:smooth}''
		Let \(f \in \Gamma_0(\bR^n)\) such that \(f=\env* g\) for some function $\func{g}{\bR^n}{\exR}$. By properness of $f$ we have that $f >-\infty$ and there is some $x \in \bR^n$ such that $\infty >f(x)\defeq\inf_{y\in \bR^n}g(y)+\D(x,y)$ implying that $\infty > g(y)$ for some $y \in \bR^n$ and thus $\infty > g(y)+\D(\cdot,y)\geq f$. This implies that $f$ is finite-valued and thus $-f=-\env* g$ is proper.
		Thanks to \cref{thm:B-Phi} we have that $-f$ is \B-weakly convex and, due to convexity and properness of $f$, that $f$ is \B-smooth.
		\qedhere
	\end{itemize}
\end{proof}

\subsection{Anisotropic convexity and smoothness}\label{sec:aniso}\leavevmode
The motivation of this work stems from the observation that a conjugate duality between \B-smoothness and \B*-strong convexity in general does not hold, as the following counterexample reveals.
	\begin{example} \label{ex:counter_duality_bregman}
		Let $\phi(x)=\frac{1}{3}|x|^3$. Choose $f(x) = \phi(x) + x$. Clearly, $f$ is both \B-strongly convex and \B-smooth. The convex conjugate $f^*$ amounts to $f^*(x)= \phi^*(x-1)$. A simple calculation shows that the second order derivative $(f^*-\phi^*)''(x)< 0$ is not bounded from below for $x$ near $0$ and $(\phi^*-f^*)''(x)<0$ is not bounded from below for $x$ near $1$. As a consequence $f^*$ is neither \B*-smooth nor \B*-strongly convex.
	\end{example}
	Instead, we shall introduce different notions of relative smoothness and strong convexity which we call \emph{anisotropic smoothness} and \emph{anisotropic strong convexity}, respectively.
The characterization of \(\Phi\)-convexity in \cref{thm:B-Phi:sup,thm:B-smooth:inf} related the notions of \B-strong convexity and \B-smoothness to expressions involving pointwise suprema and infima over a family of \emph{tilt}-parametrized functions \(q_i\coloneqq\phi+\innprod{y_i}{{}\cdot{}}+\beta_i\), for some \((y_i,\beta_i)\in\R^n\times\Rinf\) and with \(i\) ranging over some index set \(\mathcal{I}\). The calculus of \cite[Theorem 11.23(d)]{RoWe98} reveals that convex conjugates of such functions are to be found in the space of infima and suprema over \emph{shift}-parametrized families \(q_i^*=\phi^*({}\cdot{}-y_i)-\beta_i\).
This observation motivates the study of \(\Phi\)-convexity with a coupling \(\Phi(x,y)=r\phi(x-y)\) for \(r\in\set{\pm1}\).

\begin{remark}[\(\Phi\)-convexity and supremal convolutions]\label{rem:supconv}%
	As detailed in \cref{thm:Phicvx}, for any coupling \(\Phi\) it holds that a function is (left or right) \(\Phi\)-convex if and only if it is a (right or left) \(\Phi\)-conjugate.
	When \(\Phi(x,y)=r\phi(x-y)\) for \(r\in\set{\pm1}\), the right conjugate of a function \(\func g{\R^n}{\Rinf}\) amounts to the \emph{supremal convolution} of $r\phi$ and $-g$
	\begin{align}
		\conj*g(x)
	{}={} &
		\sup_{y \in \bR^n} r\phi(x - y) - g(y)\eqqcolon ((-g) \supconv r\phi)(x),
	\shortintertext{%
		and similarly the left conjugate of \(\func f{\R^n}{\Rinf}\) amounts to
	}
		\conj f(y)
	{}={} &
		\sup_{x \in \bR^n} r\phi(x - y) - f(x)= \sup_{x \in \bR^n} r\phi_-(y-x) - f(x) = ((-f) \supconv r\phi_-)(y).
	\end{align}
\end{remark}

As we shall see, while the conjugate of a \B-smooth function is a supremal convolution, the conjugate of a supremal convolution is not necessarily \B-smooth.
Instead, the anticipated duality of \B-smooth and \B-strongly convex functions involves their \emph{an\-isotrop\-ic} counterparts, defined next, which in general constitute proper subclasses of the classes of \(\Phi\)-convex functions with \(\Phi(x,y)=\pm\phi(x-y)\), as will be demonstrated in \cref{ex:counter_rel_str_cvx}.
\begin{mybox}
	\begin{definition}[{anisotropic weak/strong convexity and smoothness}]\label{defin:a-}%
		A proper lsc function $\func{f}{\bR^n}{\exR}$ is said to be
		\begin{enumerate}
		\item\label{def:a-weak}%
			\emph{\a-weakly convex} if for every \((\bar x,\bar v)\in\graph\OLDpartial f\) it holds that
			\begin{equation}\label{eq:a-weak}
				f(x) \geq f(\bar x)-\phi(x-\bar x + \nabla\phi^*(-\bar v)) + \phi(\nabla\phi^*(-\bar v)) \quad \forall x\in\R^n.
			\end{equation}
		\item\label{def:a-strong}%
			\emph{\a-strongly convex} if for every \((\bar x,\bar v)\in\graph\OLDpartial f\) it holds that
			\begin{equation}\label{eq:a-strong}
				f(x) \geq f(\bar x)+\phi(x-\bar x + \nabla\phi^*(\bar v)) - \phi(\nabla\phi^*(\bar v)) \quad \forall x\in\R^n.
			\end{equation}
		\item\label{def:a-smooth}%
			\emph{\a-smooth} if $f \in \mathcal{C}^1(\bR^n)$ and both \(f\) and \(-f\) are \a-weakly convex.
		\end{enumerate}
	\end{definition}
\end{mybox}
Note that \a-weak convexity is a globalized version of anisotropic prox-regularity introduced in \cite[Definition 2.13]{laude2021lower}.
Geometrically, the \a-weak convexity inequality \eqref{eq:a-weak} (resp. the \a-strong convexity inequality \eqref{eq:a-strong}) ensures that every limiting subgradient $\bar v \in \OLDpartial f(\bar x)$ generates a lower approximation to $f$ that supports the graph of $f$ at $\bar x$ and takes the form $x\mapsto -\phi(x-\bar y) -\bar \beta$ for $\bar y = \bar x - \nabla\phi^*(-\bar v)$ and $\bar \beta = - \phi(\nabla\phi^*(-\bar v))-f(\bar x)$ (resp. $x \mapsto \phi(x-\bar y) -\bar \beta$ for $\bar y = \bar x - \nabla\phi^*(\bar v)$ and $\bar \beta =\phi(\nabla\phi^*(\bar v))-f(\bar x)$).

Examples will be presented in \cref{sec:duality} alongside the anticipated conjugate dualities.
As will be shown in \cref{sec:gap}, another source of examples in the univariate case is the class of pointwise maxima/minima where the index sets in the definition of $\Phi$-convexity are finite.

	\begin{remark} \label{rem:a_smoothness_euclidean}
	As in the case of \B-convexity, and specifically in the characterization of \cref{thm:B-:ineq}, the definition of \a-weak and \a-strong convexity can be combined into a single inequality up to a sign difference.
Namely, a proper and lsc function \(\func{f}{\R^n}{\Rinf}\) is \a-weakly convex (resp. \a-strongly convex) if
for every \((\bar x,\bar v)\in\graph\OLDpartial f\) it holds that
\begin{equation}\label{eq:a-ineq}
	f(x)
{}\geq{}
	f(\bar x) + r\phi(x-\bar x + \nabla\phi^*(r\bar v)) - r\phi(\nabla\phi^*(r\bar v)) \quad \forall x \in \bR^n
\end{equation}
with \(r=-1\) (resp. \(r=1\)).
	In particular, if $\phi=\frac{1}{2\lambda }\|\cdot\|^2$ is a positive multiple of the squared Euclidean norm with $\lambda >0$, by expanding the square in \cref{eq:a-ineq}, we obtain
	\begin{align*}
	f(x) &\geq f(\bar x) + \langle \bar v, x - \bar x \rangle +\frac{r}{2\lambda}\|x - \bar x\|^2,
	\end{align*}
	so that the inequalities specialize to the Euclidean \emph{weak} ($r=-1$) and \emph{strong} ($r=+1$) convexity inequalities, respectively, and thus also coincide with the Bregman version \cref{eq:B-ineq} due to the identity $\frac{r}{2\lambda}\|x - \bar x\|^2=r\D(x, \bar x)$.
	\end{remark}

\begin{remark}[anisotropic descent inequality and algorithmic implications]\label{rem:anisoalg}%
	In the same way that Euclidean weak convexity of \(-f\) entails a quadratic upper bound in the likes of the Euclidean descent lemma, see, e.g., \cite[Propostion A.24]{bertsekas1995nonlinear}, \a-weak convexity of \(-f\) for $f \in \mathcal{C}^1(\bR^n)$ can be read as the \emph{anisotropic descent inequality}
	\begin{equation}\label{eq:a-descent}
		f(x) \leq f(\bar x) + \phi(x-\bar x + \nabla\phi^*(\nabla f(\bar x)) - \phi(\nabla\phi^*(\nabla f(\bar x))) \quad \forall x\in\R^n,
	\end{equation}
	holding at any \(\bar x \in \bR^n\).
	This bound is simplified if one considers $\phi$ such that $\phi(0)=0$ and $\nabla \phi(0)=0$.
	Given a function $f$ satisfying \cref{eq:a-descent} and a current iterate $x^t \in \bR^n$, by minimizing the upper bound \cref{eq:a-descent} at $\bar x = x^t$ we recast a recent anisotropic generalization of gradient descent, called \emph{dual space preconditioning} \cite{maddison2021dual},
	\begin{align*}
	x^{t+1} &:= \argmin_{x \in \bR^n} f(x^t) + \phi(x-x^t + \nabla\phi^*(\nabla f(x^t)) - \phi(\nabla\phi^*(\nabla f(x^t))) \\
	&= x^t - \nabla\phi^*(\nabla f(x^t)),
	\end{align*}
	where the second equality follows from the fact that $\nabla \phi(0)=0$ together with the first-order optimality condition $0 = \nabla \phi(x^{t+1}-x^t + \nabla\phi^*(\nabla f(x^t))$.
	Indeed, inserting this recurrence into \cref{eq:a-descent} generates a sufficient descent
	\begin{align*}
	f(x^{t+1}) -f(x^t) &\leq \phi(0) - \phi(\nabla\phi^*(\nabla f(x^t)))= - \phi(\nabla\phi^*(\nabla f(x^t)))= -\phi(x^{t}-x^{t+1}),
	\end{align*}
	where the second equality holds since $\phi(0)=0$.
	Note that this argument differs from the one in \cite{maddison2021dual}, where the descent is obtained by using a different criterion (see \cite[Lemma 3.8]{maddison2021dual}).
	In contrast to \cite{maddison2021dual} which focuses on the convex case only, the anisotropic descent inequality also holds for possibly nonconvex smooth functions.
\end{remark}

	\begin{lemma} \label{thm:lem_astrongconvexity_impl_convexity}%
		Let \(\func{f}{\R^n}{\Rinf}\) and consider the following statements:
		\begin{enumerate}
		\item\label{thm:=>:a-strong}%
			\(f\) is \a-strongly convex.
		\item\label{thm:=>:cvx}%
			\(f\) is proper, lsc and convex.
		\item\label{thm:=>:a-weak}%
			\(f\) is \a-weakly convex.
		\end{enumerate}
		Then, it holds that \(\ref{thm:=>:a-strong} \Rightarrow \ref{thm:=>:cvx} \Rightarrow \ref{thm:=>:a-weak}\).
	\end{lemma}
\begin{proof}
		If \(f\) is \a-strongly convex, then it is proper and lsc by definition.
		For any \((x_1,v_1),(x_2,v_2)\in\graph\OLDpartial f\) it follows from \eqref{eq:a-strong} that
		\begin{align*}
			f(x_2)
		&\geq
			f(x_1)+\phi(x_2-x_1+\nabla\phi^*(v_1)) - \phi(\nabla\phi^*(v_1))
		{}\geq{}
			f(x_1)+\innprod{v_1}{x_2-x_1},
		\shortintertext{%
			where the second inequality owes to convexity of \(\phi\) \cite[Theorem 2.14]{RoWe98} and the fact that \(\nabla\phi(\nabla\phi^*(v_1))=v_1\).
			Similarly,
		}
			f(x_1)
		&\geq
			f(x_2)+\innprod{v_2}{x_1-x_2}.
		\end{align*}
		By summing the two inequalities it readily follows that the mapping \(\OLDpartial f\) is monotone, and convexity of \(f\) then follows from \cite[Theorem 12.17]{RoWe98}.
		\Par
		Suppose now that \(f\) is proper, lsc and convex, and let \((\bar x,\bar v)\in\graph\OLDpartial f\).
		Then,
		\[
			f(x)
		{}\geq{}
			f(\bar x)+\innprod{\bar v}{x-\bar x}
		{}\geq{}
			f(\bar x)-\phi(x-\bar x + \nabla\phi^*(-\bar v)) + \phi(\nabla\phi^*(-\bar v))
		\quad
			\forall x\in\R^n,
		\]
		where again we used convexity of \(\phi\) and the fact that \(\nabla\phi(\nabla\phi^*(-\bar v))=-\bar v\). This shows that $f$ is \a-weakly convex.
\end{proof}

As explained above, the anisotropic subgradient inequality \eqref{eq:a-ineq} ensures that every classical subgradient $\bar v \in \OLDpartial f(\bar x)$ generates a supporting lower approximation $x\mapsto r\phi(x-\bar y) -\bar \beta$ to $f$. This means that $\bar y=\bar x - \nabla\phi^*(r\bar v)\in \subdiff f(\bar x)$ is a $\Phi$-subgradient of $f$ at $\bar x$.
As we shall see, in this case the converse is also true, that is, every $\Phi$-subgradient $\bar y\in \subdiff f(\bar x)$ takes the form $\bar y\in (\id - \nabla\phi^*\circ r\OLDpartial f)(\bar x)$.

\begin{proposition}\label{thm:subset_phi_resolvent_inv}%
	Let \(r=-1\) (resp. \(r=+1\)) and let \(\Phi(x,y)=r\phi(x-y)\).
	For any proper lsc $\func{f}{\bR^n}{\exR}$ and $\bar x \in \bR^n$ it holds that \(\subdiff f(\bar x) \subseteq \bar x - \nabla\phi^*(r\widehat\OLDpartial f(\bar x))\).
	In addition, the following conditions are equivalent:
	\begin{enumerateq}
	\item\label{thm:phi_resolvent_inv:aniso_ineq}%
		$f$ satisfies the \a-weak convexity inequality \eqref{eq:a-weak} (resp. the \a-strong convexity inequality \eqref{eq:a-strong}) at $\bar x$ for every $\bar v \in \OLDpartial f(\bar x)$;
	\item\label{thm:phi_resolvent_inv:subdiff}%
		$\subdiff f(\bar x) = \bar x - \nabla\phi^*( r\OLDpartial f(\bar x))$.
	\end{enumerateq}
	Under any of the above equivalent conditions, one has that \(\OLDpartial f(\bar x)=\widehat\OLDpartial f(\bar x)\).
	In particular, if \(f\) is either \a-weakly or \a-strongly convex, then \(\OLDpartial f=\widehat\OLDpartial f\) and \(\dom\subdiff f=\dom\widehat\OLDpartial f\).
\end{proposition}
\begin{proof}
	Let a pair $(\bar x,\bar y)\in\graph\subdiff f$ be fixed.
	In view of \cref{thm:phi_subgrad_existence} we know that $\bar x\in\argmin f-r\phi({}\cdot{}- \bar y)$, hence
	$
		0 \in\widehat\OLDpartial f(\bar x) - r\nabla \phi(\bar x - \bar y)
	$,
	and therefore $\bar y\in \bar x - \nabla \phi^* (r\widehat\OLDpartial f(\bar x)){}\subseteq \bar x - \nabla\phi^*(r\OLDpartial f(\bar v))$.
	In particular, if assertion \ref{thm:phi_resolvent_inv:subdiff} holds, then
		\[
			\bar x-\nabla\phi^*(r\OLDpartial f(\bar x))
		{}={}
			\subdiff f(\bar x)
		{}\subseteq{}
			\bar x-\nabla\phi^*(r\widehat\OLDpartial f(\bar x))
		{}\subseteq{}
			\bar x-\nabla\phi^*(r\OLDpartial f(\bar x)),
		\]
		which by the injectivity of \(\nabla\phi^*\) yields the claimed identity \(\widehat\OLDpartial f(\bar x)=\OLDpartial f(\bar x)\).
	In what follows, fix $(\bar x,\bar v)\in\graph\OLDpartial f$
	and let $\bar y\coloneqq\bar x - \nabla\phi^*(r\bar v)$.
	\begin{itemize}
	\item``\ref{thm:phi_resolvent_inv:aniso_ineq} \(\Rightarrow\) \ref{thm:phi_resolvent_inv:subdiff}''
		Substituting $\bar y$ into the anisotropic subgradient inequality \eqref{eq:a-ineq} gives
		$$
			f(x) \geq r\phi(x-\bar y) - r\phi(\bar x - \bar y) + f(\bar x)
		\quad
			\forall x\in\R^n,
		$$
		that is, $\bar y\in \subdiff f(\bar x)$. In combination with $\subdiff f(\bar x) \subseteq \bar x - \nabla \phi^* (r\OLDpartial f(\bar x))$ we obtain the desired result.
	\item``\ref{thm:phi_resolvent_inv:subdiff} \(\Rightarrow\) \ref{thm:phi_resolvent_inv:aniso_ineq}''
		By assumption, $\bar y \in \subdiff f(\bar x)$, and therefore
		$$
		f(x) \geq r\phi(x-\bar y) - r\phi(\bar x - \bar y) + f(\bar x)=r\phi(x-\bar x + \nabla\phi^*(r\bar v)) - r\phi(\nabla\phi^*(r\bar v)) + f(\bar x)
		$$
		follows from the $\Phi$-subgradient inequality \eqref{eq:subdiff}.
	\qedhere
	\end{itemize}
\end{proof}

The relation $\subdiff f = \id - \nabla\phi^*\circ r\OLDpartial f$ sets up a certain one-to-one correspondence between limiting subgradients $\bar v \in \OLDpartial f(\bar x)$ and $\Phi$-subgradients $\bar y = \bar x -\nabla \phi^*(r\bar v)$, and thus the anisotropic subgradient inequality can be understood in terms of a $\Phi$-subgradient inequality.
The next result shows that for $\Phi$-convex $f$ the $\Phi$-subdifferential is nonemp\-ty and compact at points at which $f$ is finite and strictly continuous.

\begin{proposition}\label{thm:phi_subdiff_nonempty}%
	Let $r\in\set{\pm1}$ and choose $\Phi(x,y)=r\phi(x-y)$.
	Let $\func{f}{\bR^n}{\exR}$ be $\Phi$-convex and strictly continuous at $\bar x \in \bR^n$, a point where $f$ is finite.
	Then $\subdiff f(\bar x)$ is nonempty and compact.
\end{proposition}
\begin{proof}
	Since $f$ is $\Phi$-convex, by \cref{thm:Phicvx} we have that $f(\bar x)=\sup_{y \in \bR^n} \Phi(\bar x, y) - \conj f(y)$. This means that for any \(\varepsilon>0\) there exists $\bar y_\varepsilon$ such that
		\begin{align}
		- \infty &< f(\bar x) - \varepsilon \leq \Phi(\bar x, \bar y_\varepsilon) - \conj f(\bar y_\varepsilon) \leq f(\bar x) < +\infty \label{eq:inequality_ekeland_sub}.
		\end{align}
		In addition, in view of \cref{thm:FYineq} for any $x\in \bR^n$ we have that
		\begin{equation} \label{eq:inequality_ekeland_sub2}
		f(x) \geq \Phi(x, \bar y_\varepsilon)- \conj f(\bar y_\varepsilon).
		\end{equation}
	Summing the inequalites \cref{eq:inequality_ekeland_sub,eq:inequality_ekeland_sub2} yields:
	\begin{equation}\label{eq:inequality_min}
		f(\bar x) -r\phi(\bar x-\bar y_\varepsilon)\leq f(x) -r\phi(x-\bar y_\varepsilon) + \varepsilon
	\quad
		\forall x \in \bR^n.
	\end{equation}
	Then, $\bar x \in \epsargmin_{x \in \bR^n} f(x) - r\phi(x-\bar y_\varepsilon)$.
	Ekeland's variational principle with $\delta\coloneqq \sqrt{\varepsilon}$, see \cite[Proposition 1.43]{RoWe98}, yields the existence of a point
	$\bar x_\varepsilon \in\cball{\bar x}{\sqrt\varepsilon}$ with $f(\bar x_\varepsilon) - r\phi(\bar x_\varepsilon-\bar y_\varepsilon) \leq f(\bar x) - r\phi(\bar x-\bar y_\varepsilon)$ and $\bar x_\varepsilon = \argmin f -r\phi(\cdot - \bar y_\varepsilon) + \sqrt{\varepsilon}\|\cdot - \bar x_\varepsilon\|$.
	Minimality of \(\bar x_\varepsilon\) implies
	$
		0 \in \OLDpartial f(\bar x_\varepsilon) - r\nabla \phi(\bar x_\varepsilon - \bar y_\varepsilon) + \cball{0}{\sqrt\varepsilon}
	$,
	and therefore
		\begin{equation}\label{eq:yeps}
			r\nabla \phi(\bar x_\varepsilon - \bar y_\varepsilon)- u_\varepsilon\eqqcolon \bar v_\varepsilon \in \OLDpartial f(\bar x_\varepsilon)
		\end{equation}
	holds for some $u_\varepsilon$ with $\|u_\varepsilon\| \leq \sqrt\varepsilon$.
	Since $f$, as a pointwise supremum over continuous functions, is lsc and by assumption strictly continuous at $\bar x$, in view of \cite[Theorem 9.13]{RoWe98}, $x \mapsto \OLDpartial f(x)$ is locally bounded at $\bar x$. Without loss of generality, up to extracting a subnet,
	$
		\bar v_\varepsilon \to \bar v
	$
	for some \(\bar v\in\R^n\) as $\varepsilon \searrow 0$. It follows from \eqref{eq:yeps} that $\bar y_\varepsilon =\bar x_\varepsilon - \nabla \phi^*(r\bar v_\varepsilon + r u_\varepsilon)$, and due to continuity of $\nabla \phi^*$ we obtain that $\bar y_\varepsilon \to \bar y {}\coloneqq{} \bar x - \nabla \phi^*(r \bar v)$. Passing to the limit in \eqref{eq:inequality_min} yields
	$$
		f(\bar x) +r\phi(x-\bar y)-r\phi(\bar x-\bar y)\leq f(x)
	\quad
		\forall x\in\bR^n,
	$$
	and therefore $\bar y \in \subdiff f(\bar x)$.
	
	In view of \cref{thm:phi_subgrad_existence}, since $f$ is $\Phi$-convex, we have $\bar y \in \subdiff f(\bar x) \Leftrightarrow \bar x \in \subdiff f^\Phi(\bar y) \Leftrightarrow \bar y \in \argmin\set*{\conj f-r\phi(\bar x- {}\cdot{})}$ and thus $\subdiff f(\bar x)=\argmin\set*{\conj f-r\phi(\bar x- {}\cdot{})}$. Since $\conj f$ as a pointwise supremum over continuous functions is lsc, $\subdiff f(\bar x)$ is closed.
	By \cref{thm:subset_phi_resolvent_inv}, $\subdiff f(\bar x) \subseteq \bar x - \nabla\phi^*(r\OLDpartial f(\bar x))$. Since \(f\) is strictly continuous, $\OLDpartial f(\bar x)$ is bounded, and thus so is \(\nabla\phi^*(r\OLDpartial f(\bar x))\) owing to continuity of $\nabla\phi^*$. As a result, \(\subdiff f(\bar x)\) too is bounded.
\end{proof}

	\section{Anisotropic and Bregman conjugate dualities}\label{sec:duality}%
		\subsection{\B-smooth and \a*-strongly convex conjugate duality}%
In this section we will study the conjugate duality between \a-strong convexity as in \cref{def:a-strong} and \B*-smoothness for a convex function $f\in\Gamma_0(\bR^n)$.
The main result of this section is \cref{thm:duality_aniso_str_cvx}.
To begin with we show that essential strict convexity is implied by \a-strong convexity.

\begin{proposition}\label{thm:strictly_convex}%
	Let $f :\bR^n \to \exR$ be \a-strongly convex.
	Then $f$ is essentially strictly convex.
\end{proposition}
\begin{proof}
	Invoking \cref{thm:lem_astrongconvexity_impl_convexity} we know that $f$ is convex, proper lsc and in particular has a convex domain.
	Fix a nonempty convex set $K\subseteq \dom \OLDpartial f$. To avoid trivialities, we may assume that $K$ is not a singleton. Let $x,x' \in K$ with $x\neq x'$ be fixed. Choose $\tau \in (0,1)$. Since $K$ is convex, $x_\tau\coloneqq\tau x + (1-\tau) x' \in K$. By \a-strong convexity of $f$ and \cref{thm:subset_phi_resolvent_inv} 
		\begin{align} \label{eq:essential_strict_conv_a-strong_lemma_incl}
		\dom\subdiff f =  \dom \OLDpartial f \supseteq K,
		\end{align}
		and thus there exists $y_\tau \in \subdiff f(x_\tau)$. \cref{thm:phi_subgrad_existence} implies $f(x_\tau) + \conj f(y_\tau) = \phi(x_\tau - y_
		\tau)$.
	Since $\phi$ is strictly convex and thus so is $\phi(\cdot - y_\tau) - \conj f(y_\tau)$, we have
	\begin{align*}
		f(\tau x + (1-\tau) x')
	{}={} &
		\phi(\tau x + (1-\tau) x' - y_\tau) - \conj f(y_\tau)
	\\
	{}<{} &
		\tau (\phi(x- y_\tau) - \conj f(y_\tau)) + (1-\tau) (\phi(x'- y_\tau) - \conj f(y_\tau))
	\\
	{}\leq{} &
		\tau\sup_{y \in \bR^n} \{\phi(x-y) - \conj f(y)\} + (1-\tau) \sup_{y \in \bR^n} \{\phi(x'-y) - \conj f(y)\}
	\\
	{}={} &
		\tau f(x) + (1-\tau) f(x'),
	\end{align*}
	where the last equality follows by definition of $\biconj f$ and the fact that $f\equiv\biconj f$ on $\dom\subdiff f\supseteq K$, see \cref{thm:phi_subgrad_existence} and \cref{eq:essential_strict_conv_a-strong_lemma_incl}. This proves that $f$ is essentially strictly convex.
\end{proof}


We state Pshenichnyi's and Hiriart-Urruty's formula \cite{pshenichnyi1971controle,hiriart1986general} adopted from \cite[Theorem 7.1]{cabot2017envelopes}; for completeness we provide a short proof adapted from that in the given reference. We remind that \(h_-\coloneqq h(-{}\cdot{})\), and that \((h_-)^*=(h^*)_-\eqqcolon h_-^*\).
\begin{lemma}\label{thm:hiriart_urruty}
	Let $\func{g}{\bR^n}{\exR}$ and $h \in \Gamma_0(\bR^n)$. Then the following identity holds
	\begin{equation}
		\sup_{y \in \bR^n} h(x-y) - g_-(y) = (h^* - g^*)^*(x),
	\end{equation}
	where we adopt extended arithmetics {\rm\cite{moreau1966fonctionnelles}}, if necessary.
\end{lemma}
\begin{proof}
	Since $h \in \Gamma_0(\bR^n)$ we have that
	$$
	h(x-y)=h^{**}(x-y)=\sup_{v \in \bR^n} \innprod{x-y}{v} - h^*(v).
	$$
	Thus we obtain by an exchange of the order of maximization, see \cite[Proposition 1.35]{RoWe98}
	\begin{align*}
		\sup_{y \in \bR^n} h(x-y) - g(-y)
	{}={} &
		\sup_{y \in \bR^n} \sup_{v \in \bR^n} \innprod{x-y}{v} - h^*(v) - g(-y)
	\\
	{}={} &
		\sup_{v \in \bR^n} \innprod{x}{v} - h^*(v) + \sup_{y \in \bR^n} \innprod{-y}{-v} - g(y)
	\\
	{}={} &
		\sup_{v \in \bR^n} \innprod{x}{v} - h^*(v) + g^*(v)
	\end{align*}
	which equals \((h^* - g^*)^*(x)\).
\end{proof}

Next we show the main result of this section, the anticipated duality between the classes of \a-strongly convex and \B*-smooth functions.
The inclusion of the former in the superclass of \(\Phi\)-convex functions for \(\Phi(x,y)=\phi(x-y)\) will also be demonstrated.
It will become apparent that the culprit of the properness of such inclusion is to be found in a certain saddle-point property.
Further investigation on this property will be developed in \cref{sec:gap:a-strong}, where special cases guaranteeing its validity will be presented, and its sharpness in more general contexts ultimately showcased with \cref{ex:counter_rel_str_cvx}.

\begin{theorem}[\a-strong and \B*-smooth duality]\label{thm:duality_aniso_str_cvx}%
	Let $\Phi(x,y)\coloneqq \phi(x-y)$.
	For any proper $f:\bR^n \to \exR$, the following conditions are equivalent:
	\begin{enumerateq}
	\item\label{thm:anisoDuality:a-strong}%
		$f$ is \a-strongly convex;
	\item\label{thm:anisoDuality:B*-smooth}
		$f$ is lsc and convex and $f^*$ is \B*-smooth (i.e., $\phi^* - f^*\in \Gamma_0(\R^n)$);
	\item\label{thm:anisoDuality:minmax}%
		$f$ is $\Phi$-convex (equivalently, $f=(-h)\supconv\phi$ for some $h\in\Gamma_0(\R^n)$), and for all $\bar v \in \bR^n$ the following saddle-point property holds:
		\begin{equation}\label{eq:strong_duality_aniso_str_cvx}
			\sup_{x \in \bR^n} \inf*_{y\in \bR^n} \innprod{x}{\bar v} - \Phi(x,y) + \conj f(y) = \inf*_{y\in \bR^n} \sup_{x \in \bR^n} \innprod{x}{\bar v} - \Phi(x,y) + \conj f(y).
		\end{equation}
	\end{enumerateq}
	If, in addition, $f$ is essentially smooth, the saddle-point property in {\rm\ref{thm:anisoDuality:minmax}} is superfluous.
\end{theorem}
\begin{proof}~
	\begin{itemize}
	\item``\ref{thm:anisoDuality:a-strong} \(\Rightarrow\) \ref{thm:anisoDuality:B*-smooth}'' Let $f$ be \a-strongly convex. In view of \cref{thm:lem_astrongconvexity_impl_convexity} $f \in \Gamma_0(\bR^n)$. Fix $(\bar x,\bar v)\in\graph\OLDpartial f$ and define $\bar y\coloneqq\bar x -\nabla\phi^*(\bar v) \in\subdiff f(\bar x)$, where the inclusion holds by \cref{thm:subset_phi_resolvent_inv}. Invoking \cref{thm:phi_subgrad_existence} we have $f(\bar x) +\conj f(\bar y) = \phi(\bar x-\bar y)$ and $\bar x \in \subdiff \conj f(\bar y)$, where the latter means by definition that $\bar y \in \argmax \phi(\bar x - \cdot) - \conj f$. Combined, these yield
		\begin{equation}\label{eq:supremum_y}
			\sup_{y \in \bR^n}{}\phi(\bar x-y)-\conj f(y)
			{}={}
			\phi(\bar x-\bar y)-\conj f(\bar y)
			{}=f(\bar x).
		\end{equation}
			We have
			\begin{align*}
			\phi(\bar x-y)=\phi^{**}(\bar x-y)=\sup_{v \in \bR^n} \innprod{\bar x-y}{v} - \phi^*(v).
			\end{align*}
			Define $q(y):=\sup_{v \in \bR^n} h(y,v)$ for
			\begin{align*}
			h(y, v):= \innprod{\bar x-y}{v} - \phi^*(v) - \conj f(y).
			\end{align*}
			Then we can rewrite the supremum in \cref{eq:supremum_y} in terms of the joint supremum 
			\begin{align} \label{eq:supremum_y2}
			f(\bar x)=\sup_{y \in \bR^n} q(y),
			\end{align}
			where $\bar y \in \argmax q$. Since $\bar y=\bar x -\nabla\phi^*(\bar v)$ we have that $\bar v=\nabla \phi(\bar x-\bar y)$. Using the fact that $\bar v\in \OLDpartial \phi(\bar x-\bar y) \Leftrightarrow \bar x-\bar y \in \OLDpartial \phi^*(\bar v)$, this means that $\bar v \in \argmax \innprod{\bar x-\bar y}{\cdot} - \phi^*$ showing that $\bar v \in \argmax h(\bar y, \cdot)$. Define $p(v):=\sup_{y \in \bR^n} h(y,v)$. Then \cite[Proposition 1.35]{RoWe98} yields that $(\bar y, \bar v) \in \argmax h$ as well as 
			\begin{align} \label{eq:attainment_v}
			\bar v \in \argmax p.
			\end{align}
			Overall this yields
			\begin{align*}
				f(\bar x) &\overrel*[=]{\eqref{eq:supremum_y2}} \sup_{y \in \bR^n} q(y) = \sup_{v \in \bR^n} p(v) \overrel*[=]{\eqref{eq:attainment_v}} p(\bar v) \\
				&\defeq \innprod{\bar x}{\bar v} - \phi^*(\bar v) +\sup_{y \in \bR^n} \innprod{y}{-\bar v} - \conj f(y) \defeq \innprod{\bar x}{\bar v} - \phi^*(\bar v) + (\conj f)^*(-\bar v),
			\end{align*}
			which combined with the fact that $\bar v \in \OLDpartial f(\bar x)$ results in
			\[
				f^*(\bar v)
			{}={}
				\innprod{\bar x}{\bar v} -f(\bar x)
			{}={}
				\phi^*(\bar v) - (\conj f)_-^*(\bar v) \in \bR.
			\]%
		Since $f\geq \phi({}\cdot{}-\bar y)-\conj f(\bar y)$ holds by \cref{thm:FYineq}, and since $\phi$ is super-coercive, $f$ too is super-coercive and as such $\dom f^*=\bR^n = \dom\OLDpartial f^*$ in view of \cite[Proposition 2.16]{bauschke1997legendre}.
		Since $\phi^*(\bar v) -f^*(\bar v) = (\conj f)_-^*(\bar v)\in \bR$ holds for all $\bar v \in \ran \OLDpartial f = \dom \OLDpartial f^*=\bR^n$, we have that $\phi^*-f^*$ is  finite-valued, convex and lsc, and thus $\phi^* - f^*\in \Gamma_0(\R^n)$ as claimed.

	\item``\ref{thm:anisoDuality:B*-smooth} \(\Rightarrow\) \ref{thm:anisoDuality:a-strong}''
		By assumption,
			\begin{align} \label{eq:def_g}
			g\coloneqq\phi^* - f^* \in \Gamma_0(\bR^n).
			\end{align}
			It follows from \cref{thm:B-weak} that $f^* \in \mathcal{C}^1(\R^n)$, and as such so is $g$.
		Fix $(\bar x,\bar v)\in\graph\OLDpartial f$ and let $\bar y\coloneqq\bar x-\nabla\phi^*(\bar v)$, so that $\nabla \phi(\bar x-\bar y) = \bar v\in \OLDpartial f(\bar x)$.
		Since $f \in \Gamma_0(\bR^n)$ by assumption, by using the fact that $\bar v \in \OLDpartial f(\bar x) \Leftrightarrow \bar x \in \OLDpartial f^*(\bar v)$ and the smoothness of $f^*$, we obtain
			\begin{align} \label{eq:incl_dual_f}
			\bar x =\nabla f^*(\bar v) = \nabla f^*(\nabla \phi(\bar x-\bar y)).
			\end{align}
		We have
		\begin{align*}
			f(x)
		{}={}
			f^{**}(x)
		{}={}
			(\phi^* - g)^*(x)
		{}={} &
			\sup_{y \in \bR^n} \phi(x-y) - g_-^*(y)
		\\
		\numberthis\label{eq:fsupconv}
		{}\geq{} &
			\phi(x- \bar y) - g_-^*(\bar y),
		\end{align*}
		where the last equality follows from \cref{thm:hiriart_urruty}.
		Smoothness of \(g\) yields
		\begin{align*}
			\bar x \overrel*{\eqref{eq:incl_dual_f}} \nabla f^*(\nabla \phi(\bar x-\bar y))
			& \overrel*{\eqref{eq:def_g}}
			\nabla (\phi^* -g)(\nabla \phi(\bar x -\bar y))
		\\
			&= \nabla\phi^*(\nabla\phi(\bar x -\bar y)) - \nabla g(\nabla \phi(\bar x -\bar y)) \\
			&=\bar x-\bar y - \nabla g(\nabla \phi(\bar x -\bar y)),
		\end{align*}
		and therefore
		\begin{align} \label{eq:subgradient_g}
			\nabla g(\nabla \phi(\bar x-\bar y)) = -\bar y.
		\end{align}
		We thus have
		\begin{align*}
			f(\bar x)
		{}\overrel*{\eqref{eq:incl_dual_f}}{} &
			\innprod{\nabla\phi(\bar x -\bar y)}{\bar x} - f^*(\nabla \phi(\bar x-\bar y))
		\\
		{}\overrel*{\eqref{eq:def_g}} {} &
			\innprod{\nabla\phi(\bar x-\bar y)}{\bar x} - \phi^*(\nabla \phi(\bar x-\bar y)) + g(\nabla \phi(\bar x-\bar y))
		\\
		{}={} &
			\innprod{\nabla \phi(\bar x-\bar y)}{\bar x-\bar y}
			{}-{}
			\phi^*(\nabla \phi(\bar x-\bar y))
			{}-{}
			\bigl(
				\innprod{\nabla \phi(\bar x-\bar y)}{-\bar y}
				{}-{}
				g(\nabla \phi(\bar x-\bar y))
			\bigr)
		\\
		{}\overrel*{\eqref{eq:subgradient_g}}{} &
			\phi(\bar x - \bar y) - g^*(-\bar y),
		\end{align*}
		where the last equality uses the fact that $g \in \Gamma_0(\bR^n)$.
		By combining the above equality with \eqref{eq:fsupconv} we obtain via the definition of $\bar y=\bar x-\nabla\phi^*(\bar v)$ that
		\[
			f(x)
		{}\geq{}
			f(\bar x)
			{}+{}
			\phi(x- \bar y)
			{}-{}
			\phi(\bar x - \bar y)
		{}={}
			f(\bar x)
			{}+{}
			\phi(x-\bar x + \nabla\phi^*(\bar v))
			{}-{}
			\phi(\nabla\phi^*(\bar v))
		\quad
			\forall x\in\R^n,
		\]
		which is the \a-strong convexity inequality \eqref{eq:a-strong}.
	\item``\ref{thm:anisoDuality:minmax} \(\Rightarrow\) \ref{thm:anisoDuality:B*-smooth}''
		Since $f$ is $\Phi$-convex, owing to \cref{thm:Phicvx}, $f=\biconj f \defeq (-f^\Phi) \supconv \phi$. Thus, both $f$ and $h:=f^\Phi$ are pointwise suprema over convex, lsc functions and thus convex lsc. Since $f$ is proper, $h(y)=\sup_{x \in \bR^n} \phi(x - y) - f(x) > -\infty$ for all $y \in \bR^n$. Suppose that $h\equiv +\infty$. This means that for any $x\in \bR^n$ we have $f(x)=\biconj f(x) = \conj h(x)=\sup_{y \in \bR^n}\phi(x - y) - h(y) = -\infty$ which contradicts properness of $f$. Thus $h$ must be proper as well.
		For any $\bar v \in \bR^n$ it holds that
		\begin{align*}
			\phi^*(\bar v)-(\conj f)_-^*(\bar v)
		{}\defeq{} &
			\inf_{y \in \bR^n} \innprod{y}{\bar v} + \conj f(y) + \sup_{\xi \in \bR^n} \innprod{\xi}{\bar v} -\phi(\xi)
		\\
		{}={} &
			\inf_{y \in \bR^n} \sup_{\xi \in \bR^n} \innprod{\xi+y}{\bar v} + \conj f(y) -\phi(\xi)
		\\
		{}\overrel*{\(x:=\xi+y\)}{} &
		\inf*_{y\in \bR^n}\sup_{x \in \bR^n}{
				\innprod{x}{\bar v}
				{}-{}
				\phi(x-y)
				{}+{}
				\conj f(y)
			}
		\\
		{}\overrel*{\eqref{eq:strong_duality_aniso_str_cvx}} &
			\sup_{x \in \bR^n}\inf*_{y\in \bR^n}{
				\innprod{x}{\bar v}
				{}-{}
				\phi(x-y)
				{}+{}
				\conj f(y)
			}
		\\
		{}\defeq{} &
			\sup_{x \in \bR^n}{
				\innprod{x}{\bar v}
				{}-{}
				\biconj f(x)
			}
		\\
		{}\defeq{} &
			(\biconj f)^*(\bar v)
		{}={}
			f^*(\bar v),
		\numberthis\label{eq:saddle_equiv}
		\end{align*}
		where the last equality owes to \cref{thm:Phicvx}. Therefore, and since $h \in \Gamma_0(\bR^n)$ we have \(\phi^*-f^*=(\conj f)_-^*=h_-^*\in\Gamma_0(\R^n)\).
		
	\item``\ref{thm:anisoDuality:B*-smooth} \(\Rightarrow\) \ref{thm:anisoDuality:minmax}''
		Let $g\coloneqq\phi^* - f^*\in\Gamma_0(\R^n)$. We have
		$$
			f=f^{**}=(\phi^* - g)^*=\sup_{y \in \bR^n} \phi(\cdot-y) - g_-^*(y)
		{}={}
			(-g_-^*)\supconv\phi,
		$$
		where the third equality follows from \cref{thm:hiriart_urruty}, and thus $f$ is $\Phi$-convex.
		By applying \cref{thm:hiriart_urruty} again to $g= \phi^* - f^*$ we also have
		\[
			g^*=(\phi^*-f^*)^* = \sup_{x \in \bR^n} \phi(\cdot - x) - f(-x) =\sup_{x \in \bR^n} \phi(x- (-\cdot)) - f(x) \defeq (\conj f)_-.
		\]
		Since $g\in\Gamma_0(\R^n)$, by taking conjugates on both sides we obtain that $g=(\conj f)_-^*$.
		In addition it holds that
		$$
			\biconj f\defeq\sup_{y \in \bR^n} \phi(\cdot-y) - \conj f(y)=(\phi^* - (\conj f)_-^*)^*,
		$$
		where the second identity again uses \cref{thm:hiriart_urruty}.
		By taking conjugates on both sides, since $\phi^* - (\conj f)_-^*=\phi^* - g=f^*$ is convex, proper, and lsc, we obtain that $(\biconj f)^*=\phi^* - (\conj f)_-^*$.
		As derived in \eqref{eq:saddle_equiv}, this is precisely the sought saddle point identity \eqref{eq:strong_duality_aniso_str_cvx}.

	\item Suppose now that $f$ is $\Phi$-convex and essentially smooth. Let $\bar x \in \dom \OLDpartial f$. Owing to essential smoothness of $f$, in view of \cite[Theorem 26.1]{Roc70} we have $\bar x \in \interior(\dom f)=\dom \OLDpartial f$ and $f$ is differentiable at $\bar x$. In particular it is strictly continuous at $\bar x$.
		Invoking \cref{thm:phi_subdiff_nonempty} we know that 
		$$
		\emptyset \neq \subdiff f(\bar x) \subseteq \bar x - \nabla \phi^*(\OLDpartial f(\bar x)),
		$$
		where the inclusion follows from \cref{thm:subset_phi_resolvent_inv}.
		Since $\OLDpartial f(\bar x)=\set{\nabla f(\bar x)}$, the right hand side of the inclusion is single-valued. Since the left-hand side is nonempty, the inclusion holds with equality.
		\Cref{thm:subset_phi_resolvent_inv} then implies that the anisotropic strong convexity inequality \eqref{eq:a-strong} holds at every $(\bar x, \bar v) \in \graph \OLDpartial f$ and thus $f$ is \a-strongly convex.
	\qedhere
		\end{itemize}
\end{proof}

\begin{remark}
Let \(\Phi(x,y)=\phi(x-y)\) and recall the definition of the (left) \(\Phi\)-conjugate
$
	\conj f(y)=\sup_{x\in\bR^n} \phi(x - y) - f(x)
$.
If \(\phi=\frac12\|{}\cdot{}\|^2\) is	 quadratic, by expanding the square we obtain $\conj f(y)=\frac{1}{2}\|y\|^2 +\sup_{x\in\bR^n} \langle x,-y\rangle +\frac12\|x\|^2 - f(x) =\frac{1}{2}\|y\|^2 + (f-\frac{1}{2}\|\cdot\|^2)^*_-(y)$ and hence $\conj f-\frac{1}{2}\|\cdot\|^2$ is convex. This shows that the saddle-point problem \eqref{eq:strong_duality_aniso_str_cvx} becomes convex-concave with a bilinear coupling:
	$$
		f^*(\bar v) = \sup_{x\in \bR^n} \inf*_{y \in \bR^n} \innprod{x}{\bar v} - \tfrac{1}{2}\|x\|^2 + \innprod{x}{y} + \conj f(y) -\tfrac{1}{2}\|y\|^2.
	$$
As a consequence, the saddle-point property in \cref{thm:anisoDuality:minmax} is superfluous in the Euclidean case.
\end{remark}
In \cref{sec:gap:a-strong} we will illustrate the sharpness of the saddle-point property in the non-Euclidean case.

\subsection{\B-strongly convex and \a*-smooth conjugate duality}
The main result of this section is \cref{thm:anisotropic_smoothness}, which, unlike the main result of the previous subsection, is derived via an existing conjugate duality between infimal convolution and pointwise addition \cite[Theorem 11.23(a)]{RoWe98}.
We first prove the following result which shows that at least for smooth functions $\Phi$-convexity for $\Phi(x,y)=-\phi(x-y)$ and \a-weak convexity are equivalent.
As will be shown in \cref{ex:counterex_weak_nd} this equivalence does not hold in general.
\begin{proposition} \label{thm:C1:a-weak_equivalence}
	Let $\Phi(x,y)=-\phi(x-y)$, and let $f\in\mathcal{C}^1(\R^n)$.
	Then the following conditions are equivalent:
	\begin{enumerateq}
	\item \label{thm:phiconvexC1:a-weak}%
		$f$ is \a-weakly convex;
	\item \label{thm:phiconvexC1:infconv}%
		$f$ is $\Phi$-convex.
	\end{enumerateq}
\end{proposition}
\begin{proof}~
	\begin{proofitemize}
	\item``\ref{thm:phiconvexC1:a-weak} \(\Rightarrow\) \ref{thm:phiconvexC1:infconv}''
		The \a-weak convexity bound \eqref{eq:a-weak} in combination with \cref{thm:subset_phi_resolvent_inv} implies that $\subdiff f(\bar x)\neq\emptyset\) for all \(\bar x\in\R^n\).
		It then follows from \cref{thm:phi_subgrad_existence} that $\biconj f\equiv f$ on \(\R^n\), and as such \(f\) is \(\Phi\)-convex by \cref{thm:Phicvx}.
	\item``\ref{thm:phiconvexC1:infconv} \(\Rightarrow\) \ref{thm:phiconvexC1:a-weak}''
		Let $\bar x \in \bR^n$.
		Invoking \cref{thm:phi_subdiff_nonempty} for $r=-1$, we have%
		$$
			\emptyset
		{}\neq{}
			\subdiff f(\bar x)
		{}\subseteq{}
			\bar x-\nabla\phi^*(-\OLDpartial f(\bar x))
		{}={}
			\{\bar x-\nabla\phi^*(-\nabla f(\bar x)) \},
		$$
		where the inclusion follows from \cref{thm:subset_phi_resolvent_inv}.
		Since the right-hand side is single valued, the inclusion holds with equality.
		\Cref{thm:subset_phi_resolvent_inv} then implies that $f$ satisfies the \a-weak convexity bound \eqref{eq:a-weak} at $\bar x$ for all $\bar v \in \OLDpartial f(\bar x)$.
	\qedhere
	\end{proofitemize}
\end{proof}

Further restricting to convex $f$ we can eventually deduce the conjugate duality between \a-smooth and \B*-strongly convex functions invoking an existing conjugate duality between infimal convolution and pointwise addition \cite[Theorem 11.23(a)]{RoWe98}.
\begin{theorem}\label{thm:anisotropic_smoothness}%
	Let \(\Phi(x,y)=-\phi(x-y)\).
	For any $f \in \Gamma_0(\bR^n)$, the following conditions are equivalent:
	\begin{enumerateq}
	\item\label{thm:a-smooth:a-smooth}%
		\(f\) is \a-smooth;
	\item\label{thm:a-smooth:phi_convex}%
		$-f$ is $\Phi$-convex;
	\item\label{thm:a-smooth:B-str}%
		$f^*$ is \B*-strongly convex;
	\item\label{thm:a-smooth:infconv}%
		$f=g\infconv\phi$ for some $g\in \Gamma_0(\bR^n)$;
	\item\label{thm:a-smooth:a_weakly}%
		\(-f\) is \a-weakly convex.
	\end{enumerateq}
\end{theorem}
\begin{proof}~
	\begin{itemize}
	\item``\ref{thm:a-smooth:infconv} \(\Rightarrow\) \ref{thm:a-smooth:B-str}''
		Assume that $f=g \infconv\phi$ for some $g\in \Gamma_0(\bR^n)$. By taking the convex conjugate we obtain via \cite[Theorem 11.23(a)]{RoWe98} that $f^*=(g\infconv \phi)^*=g^*+ \phi^*$, showing that $f^*-\phi^*=g^*$ is convex.
	\item``\ref{thm:a-smooth:B-str} \(\Rightarrow\) \ref{thm:a-smooth:infconv}''
		Suppose that $h\coloneqq f^*- \phi^*$ is convex. Since $f^*$ is proper lsc and $\phi^*$ is smooth we deduce that $h$ is also proper and lsc. Then, since $f$ is proper, convex and lsc and $f^*=h+\phi^*$, we have that $f=f^{**}=(h+\phi^*)^*$, and since $\phi^*$ has full domain, in view of \cite[Theorem 11.23(a)]{RoWe98} we obtain that $(h+\phi^*)^*=h^* \infconv \phi$. The claim follows by considering $g=h^* \in \Gamma_0(\bR^n)$.
	\item``\ref{thm:a-smooth:a-smooth} \(\Rightarrow\) \ref{thm:a-smooth:phi_convex}''
		Let $f$ be \a-smooth. In particular this means that $-f$ is \a-weakly convex and $f \in \mathcal{C}^1(\bR^n)$. Invoking \cref{thm:C1:a-weak_equivalence} this implies that $-f$ is $\Phi$-convex.
	\item``\ref{thm:a-smooth:phi_convex} \(\Rightarrow\) \ref{thm:a-smooth:infconv}''
		In view of \cref{thm:Phicvx} $\Phi$-convexity of $-f$ is equivalent to the identity $\biconj{(-f)}=-f$. Specialized to $\Phi(x,y)=-\phi(x-y)$ we have that
	\begin{align}
	f=-\biconj{(-f)}&\defeq-\sup_{y \in \bR^n} -\phi(\cdot-y) -\conj{(-f)}(y) \notag \\
	&=\inf_{y \in \bR^n} \phi(\cdot-y) +\conj{(-f)}(y)\defeq \conj{(-f)} \infconv \phi \label{eq:inf_conv_biconj}.
	\end{align}
	It remains to show that $\conj{(-f)} \in \Gamma_0(\bR^n)$. Thanks to \cref{thm:hiriart_urruty} since $f \in \Gamma_0(\bR^n)$ we have
	\begin{align*}
	(-f)^\Phi(y) = \sup_{x \in \bR^n} -\phi(x-y) + f(x) = \sup_{\xi \in \bR^n} f(y-\xi) - \phi_-(\xi) = (f^*-\phi^*)^*(y),
	\end{align*}
	and thus, in particular, $(-f)^\Phi$ is proper, convex, lsc.
	\item``\ref{thm:a-smooth:infconv} \(\Rightarrow\) \ref{thm:a-smooth:phi_convex}'' Follows by \cref{thm:Phicvx} using the identity \cref{eq:inf_conv_biconj}.
	\item``\ref{thm:a-smooth:infconv} \(\Rightarrow\) \ref{thm:a-smooth:a-smooth}''
		Let $f=g \infconv \phi$ for some $g \in \Gamma_0(\bR^n)$. As shown above, this implies that $-f$ is $\Phi$-convex. Since $\phi \in \mathcal{C}^1(\bR^n)$ is super-coercive, \cite[Corollary 18.8]{BaCo110} implies that $f\in\mathcal C^1(\R^n)$. Thanks to \cref{thm:C1:a-weak_equivalence} we deduce that $-f$ is \a-weakly convex. Since $f \in \Gamma_0(\bR^n)$, in view of \cref{thm:lem_astrongconvexity_impl_convexity}, $f$ is \a-weakly convex too. Overall this means that $f$ is \a-smooth. 
	\item ``\ref{thm:a-smooth:a-smooth} \(\Rightarrow\) \ref{thm:a-smooth:a_weakly}''
	Follows by definition (cf. \cref{def:a-smooth})  
	\item ``\ref{thm:a-smooth:a_weakly} \(\Rightarrow\) \ref{thm:a-smooth:phi_convex}'' Assume that $-f$ is \a-weakly convex. This implies that $-f$ is proper by definition. Since $f$ is proper convex, this means that $f$ is finite-valued, and thus strictly continuous in view of \cite[Example 9.14]{RoWe98}. This implies that also $-f$ is strictly continuous and thus, in view of \cite[Theorem 9.13]{RoWe98} and \cite[Corollary 8.10]{RoWe98}, $\dom \OLDpartial(-f) = \bR^n$. \A-weak convexity of $-f$ thus implies via \cref{thm:subset_phi_resolvent_inv} that $\dom \subdiff (-f) = \bR^n$. In view of \cref{thm:phi_subgrad_existence} this means that $-f = \biconj{(-f)}$ and hence $\Phi$-convex by \cref{thm:Phicvx}.
	\qedhere
	\end{itemize}
\end{proof}
The equivalence between \cref{thm:a-smooth:B-str,thm:a-smooth:infconv} was previously shown in \cite[Lemma 4.2]{wang2021bregman} for more general Legendre functions $\phi$ that need not be super-coercive.

\subsection{Examples and \a-smooth calculus} \label{sec:examples}
	We conclude this section by providing examples of \a-smooth and \a-strongly convex functions and their calculus. 
	By exploiting the conjugate duality between \B-smooth and \a*-strongly convex functions from \cref{thm:duality_aniso_str_cvx}, we provide an example for an \a*-strongly convex function on $\bR^2$.
		\begin{example} \label{ex:a_strongly_cvx}
			Let $A=[1~-1] \in \bR^{1 \times 2}$ and $b=5$.
			Consider $\func{f}{\bR^2}{\bR}$ defined by $f(x)=\frac14|Ax-b|^4$.
			In view of \cite[Proposition 2.1]{lu2018relatively}, $f$ is \B-smooth with reference function
			$$
				\phi(x)=\tfrac L4 \|x\|_2^4 + \tfrac L 2\|x\|_2^2
			$$
			for $L=3\|A\|^4+6\|A\|^3 |b| + 3\|A\|^2|b| = 42 + 60\sqrt 2$. Invoking \cite[Theorem 11.23(b)]{RoWe98} we have
			\begin{align*}
				f^*(y) &= \inf\{ (3/4)|x|^{4/3} + bx : A^\top x = y\} =
				\begin{ifcases}
					5 y_1 +\frac34 |y_1|^{4/3} & y_1 + y_2=0 \\
					\infty\otherwise.
				\end{ifcases}
			\end{align*}
			In light of \cref{thm:duality_aniso_str_cvx}, $f^*$ is \a*-strongly convex, where
			$$
				\phi^*(y)=\tfrac{3}{4L^{1/3}}  \|y\|_{2}^{4/3} \infconv \tfrac{1}{2L}\|y\|_2^2.
			$$
		\end{example}
	Next we provide an example for a function which is \a-smooth but, due to the lack of shift-invariance in Bregman distances, not \B-smooth. This complements \cref{ex:counter_duality_bregman}.

	\begin{example}[\a-smoothness versus \B-smoothness]\label{ex:shift_vs_tilt}%
		Let $\phi=\tfrac23|{}\cdot{}|^{\nicefrac32}$, and let $\func{f}{\bR}{\bR}$ be defined as
		\(
			f
		{}={}
			\indicator_{[-1,1]} \infconv \phi
		\),
		namely
		$$
			f(x)
		{}={}
			\begin{ifcases}
				0 & x \in [-1,1] \\
				(2/3)|x + 1|^{3/2} & x < -1 \\
				(2/3)|x - 1|^{3/2} & x > 1.
			\end{ifcases}
		$$
		By \cref{thm:anisotropic_smoothness}, $f$ is \a-smooth.
		The convex conjugate $f^*=|{}\cdot{}| + (1/3)|{}\cdot{}|^{3}$ is \B*-strongly convex, where $\phi^*=\frac13|{}\cdot{}|^3$.
		We prove that $f$ is not {\def\kernel{L\phi}\B}-smooth for any $L>0$. Note that $f''(x)= 1/\sqrt{4(x - 1)}$ for $x > 1$ while $L\phi''(x) = L/\sqrt{4 x}$ for $x > 0$.
		For $x \searrow 1$ we have that $L\phi''(x) \searrow L/2$ while $f''(x) \to \infty$.
		This contradicts convexity of $L\phi-f$ which would require $L\phi''-f'' \geq 0$ on $(1, \infty)$.

		Moreover, $f$ is not {\def\kernel{L\phi({}\cdot{}-b)}\B}-smooth for any $L>0$ and $b \in \bR^n$, since $f$ has an unbounded second-order derivative at the points $x\in \{\pm1\}$, i.e., $f''(x) \to \infty$ for $x\nearrow -1$ resp. $x \searrow +1$.
	\end{example}

	By invoking \cref{thm:anisotropic_smoothness} we can provide a simple rule of calculus for \a-smooth functions.
For $a>0$ and some function $\func{f}{\bR^n}{\exR}$ the epi-scaling $a \star f$ is defined by $(a \star f)(x):=a f(a^{-1} x)$.
	\begin{corollary}[epi-calculus for \a-smoothness]%
		Let $f_1,f_2 \in \Gamma_0(\bR^n)$ be \a-smooth relative to $\phi_1$ and $\phi_2$ respectively, and let $a_1, a_2>0$. Assume that one function is lower bounded and the other is coercive, or that one of the two functions is super-coercive.
		Then, $a_1 \star f_1\infconv a_2 \star f_2$ is \a-smooth with reference function $\phi:=a_1 \star \phi_1 \infconv a_2 \star \phi_2$.
	\end{corollary}
	\begin{proof}
		In view of \cref{thm:anisotropic_smoothness} there exist $g_i\in\Gamma_0(\bR^n)$ such that $f_i=g_i\infconv\phi_i$, $i=1,2$. By \cite[Proposition 12.14]{BaCo110} we have that $a_1 \star f_1\infconv a_2 \star f_2 \in \Gamma_0(\bR^n)$ and thus
		\begin{align*}
		a_1 \star f_1\infconv a_2 \star f_2 &=(a_1 \star f_1\infconv a_2 \star f_2)^{**} =(a_1 f_1^* + a_2 f_2^*)^* \\
		&=( a_1 g_1^*+ a_2 g_2^* + a_1 \phi_1^* + a_2 \phi_2^*)^*.
		\end{align*}
		Since $\phi_i$ is Legendre, super-coercive and has full domain, this also holds for $a_1 \phi_1^* + a_2 \phi_2^*$ which is thus compliant with \cref{assum:a1,assum:a2,assum:a3}, and by \cite[Theorem 11.23(a)]{RoWe98} $(a_1 \phi_1^* + a_2 \phi_2^*)^*=a_1 \star \phi_1 \infconv a_2 \star \phi_2=\phi$.
		Therefore, $a_1 g_1^*+ a_2 g_2^* + a_1 \phi_1^* + a_2 \phi_2^*$ is \B*-strongly convex. Invoking \cref{thm:anisotropic_smoothness} $( a_1 g_1^*+ a_2 g_2^* + a_1 \phi_1^* + a_2 \phi_2^*)^* = a_1 \star f_1\infconv a_2 \star f_2$ is \a-smooth as claimed.
	\end{proof}

	In contrast to \B-smoothness, the sum of \a-smooth functions may fail to be \a-smooth. In fact, \a-smoothness is not even preserved by addition of a linear function.
	In combination with \cref{ex:counter_duality_bregman} this illustrates that in general shift- and tilt-invariance are mutually exclusive properties of \a- and \B-smoothness respectively.
	\begin{example}[lack of tilt-invariance for \a-smooth functions]%
		Let $f=\phi=\frac13|{}\cdot{}|^3$ and $g(x)=5x$.
		Clearly, both $f$ and $g$ are \a-smooth, since $g^*=\indicator_{\set 5}$ is \B*-strongly convex, where $\phi^*=\frac23|{}\cdot{}|^{\nicefrac 32}$.
		In view of \cite[Theorem 11.23(a)]{RoWe98} we have
		$$
			f+g=(f+g)^{**}=(f^* \infconv g^*)^*,
		$$
		where
		$f^* \infconv g^* = \tfrac23|{}\cdot{}|^{3/2} \infconv \indicator_{\set 5} = \tfrac23|{}\cdot{} - 5|^{3/2}$.
		For $x \searrow 0$ we have $(f^* \infconv g^*)''(x) \to 1/\sqrt{20}$, while $(\phi^*)''(x)= 1/\sqrt{4 x} \to \infty$. Thus $f^* \infconv g^*$ is not \B*-strongly convex.
		From \cref{thm:anisotropic_smoothness} we infer that $f+g$ is not \a-smooth.
	\end{example}
	If the conjugate reference function $\phi^*$ generates a jointly convex Bregman distance $\D*$, \a-smoothness is closed under pointwise average as recently shown in the context of the Bregman proximal average \cite[Theorem 5.1(ii)]{wang2021bregman}. Unfortunately, however, joint convexity of $\D*$ under \cref{assum:a3}, i.e., full domain of $\phi^*$, implies that $\phi^*$ is quadratic \cite[Remark 3.6]{bauschke2001joint}. The more general case has been recently explored in \cite{laude2022anisotropic} along with a practical algorithm and more practical examples complementing the ones provided in \cite{maddison2021dual}.

	\section{Anisotropic and generalized convexity gap}\label{sec:gap}%
		\subsection{Univariate pointwise maxima and minima}
In this section we investigate whether $\Phi$-convexity implies anisotropic convexity, i.e., whether $\Phi$-convex functions for $\Phi(x,y)=r\phi(x-y)$ with $r \in \set{\pm1}$ satisfy the anisotropic subgradient inequality \eqref{eq:a-ineq}.
We have already seen in \cref{thm:C1:a-weak_equivalence} (resp. \cref{thm:duality_aniso_str_cvx}) that $\Phi$-convexity is equivalent to \a-weak convexity (resp. \a-strong convexity) under (essential) smoothness of $f$. Instead, in this section our focus is on nonsmooth $\Phi$-envelopes. To this end we confine ourselves to $\Phi$-convex functions where the index set $\mathcal{I}$ as in \cref{defin:Phicvx} is finite. This, in general, leads to functions which have downwards pointing cusps at points at which multiple pieces intersect and the limiting subdifferential is multivalued.
We first show that at least in the univariate case such functions satisfy the anisotropic subgradient inequality.
	To this end we first show that this holds whenever the index set $\mathcal{I}$ is a doubleton. The situation is depicted in \cref{fig:lower_bounds_failure} in the upper row.
	\begin{lemma} \label{thm:subgradient_pointwise_max}%
		Let $n=1$.
		Let $r =+1$ (resp. $r=-1$) and choose $\Phi(x,y)=r\phi(x-y)$.
		Let $f:\bR \to \bR$ with $f(x)=\max_{i \in \{1,2\}} \Phi(x,y_i) -\beta_i$ for some $y_i,\beta_i \in \bR$ with $i \in \{1,2\}$. Then $f$ satisfies the \a-strong convexity inequality \eqref{eq:a-strong} (resp. the \a-weak convexity inequality \eqref{eq:a-weak}) at any $(\bar x, \bar v) \in \graph \OLDpartial f$.
	\end{lemma}
	\begin{proof}
			Let $\bar x \in \bR$ be fixed and define $h_i \coloneqq  r\phi(\cdot - y_i) - \beta_i$ for some $y_i,\beta_i \in \bR$.
			Fix $\bar v \in \OLDpartial f(\bar x)=\con\set{h_i'(\bar x)}[h_i(\bar x) = f(\bar x),~i \in \{1,2\}]$, cf. \cite[Exercise 8.31]{RoWe98}.
			If \(\bar v=h_j'(\bar x)\) for some \(j\in\set{1,2}\) with $r\phi(\bar x - y_j) -\beta_j\defeq h_j(\bar x)=f(\bar x)$, the anisotropic subgradient inequality \eqref{eq:a-ineq} holds at $(\bar x, y_j)$ having
			$$
				f(x)\defeq\max\set{r\phi(x - y_i) - \beta_i}[i \in \{1,2\}] \geq f(\bar x) + r\phi(x - y_j) - r\phi(\bar x - y_j),
			$$
			in this case. Therefore, we may assume that $h_1(\bar x) = h_2(\bar x)=f(\bar x)$ and $\bar v = \tau h_1'(\bar x) + (1-\tau)h_2'(\bar x)$ for some $\tau\in(0,1)$.
			Let $\bar y\coloneqq\bar x - (\phi^*)'(r\bar v)$.
			To arrive at a contradiction, suppose that there exists $\hat x$ at which the anisotropic subgradient inequality \eqref{eq:a-ineq} is violated:
			$$
				r\phi(\hat x - y_i) - \beta_i < f(\bar x) + r\phi(\hat x-\bar y) - r\phi(\bar x -\bar y),
			\quad
				i \in \{1,2\}.
			$$
			Since $\beta_i = r\phi(\bar x - y_i) - f(\bar x)$, the above inequality can be equivalently rewritten as
			$$
				q_{\bar y}(\hat x + (\bar y- y_i)) - q_{\bar y}(\hat x) <  q_{\bar y}(\bar x+ (\bar y- y_i)) - q_{\bar y}(\bar x),
			\quad
				i =1,2,
			$$
			where $q_{\bar y}\coloneqq r\phi({}\cdot{}- \bar y)$.
			Without loss of generality assume that $r=1$.
			Since $q_{\bar y} \in \mathcal{C}^1(\bR^n)$ is strictly convex (strictly concave if $r=-1$) and thus $(q_{\bar y}'(x+(\bar y- y_i))-q_{\bar y}'(x))(\bar y- y_i) > 0$, the function $x \mapsto q_{\bar y}(x+(\bar y- y_i))-q_{\bar y}(x)$ is strictly monotone (either increasing or decreasing depending on the sign of $\bar y- y_i$).
			The above inequality then implies that
			$$
				q_{\bar y}(x + (\bar y- y_i)) - q_{\bar y}(x) <  q_{\bar y}(\bar x+ (\bar y- y_i)) - q_{\bar y}(\bar x),
			\quad
				i \in \{1,2\},
			$$
			holds for all $x$ between $\hat x$ and $\bar x$.
			This means that $f(x) < r\phi(x-\bar y) - r\phi(\bar x - \bar y) + f(\bar x)$,
			for all such $x$, and with equality holding for $x=\bar x$.
			Without loss of generality, assume that $\hat x > \bar x$.
			Let $j\in\set{1,2}$ be such that $h_j'(\bar x) > \bar v$ (which exists since $\bar v$ lies strictly between $h_1'(\bar x)$ and $h_2'(\bar x)$).
			Since $\bar v = r\phi'(\bar x - \bar y)$, by linearizing $r\phi(\cdot-\bar y)$ at $\bar x$ one obtains
			\begin{align*}
				f(x)
			{}<{} &
				r\phi(x-\bar y) - r\phi(\bar x - \bar y) + f(\bar x)
			{}={}
				\bar v(x-\bar x) + o(|x-\bar x|) + f(\bar x)
			\shortintertext{%
				for all \(x\in(\bar x,\hat x)\).
				Similarly, by linearizing $r\phi(\cdot-y_j)$ at $\bar x$ we also have
			}
				f(x)
			{}\geq{} &
				r\phi(x-y_j) - r\phi(\bar x - y_j) + f(\bar x)
			{}={}
				h_j'(\bar x) (x-\bar x) + o(|x-\bar x|) + f(\bar x)
			\end{align*}
			for all \(x\in(\bar x,\hat x)\).
			By combining the two inequalities we arrive to
			$$
				h_j'(\bar x) (x-\bar x) + o(|x-\bar x|) <  \bar v (x-\bar x) + o(|x-\bar x|),
			$$
			whence the contradiction \(h_j'(\bar x)<\bar v\).
	\end{proof}

\begin{proposition} \label{thm:subgradient_1d}%
	Let $n=1$. Let $r=+1$ (resp. $r=-1$) and choose $\Phi(x,y)=r\phi(x-y)$.
	Let $f:\bR \to \bR$ with $f(x)=\max_{i \in \mathcal{I}} \Phi(x,y_i) -\beta_i$ for some $y_i,\beta_i \in \bR$ with $i \in \mathcal{I}$ finite.
	Then $f$ is \a-strongly convex (resp. \a-weakly convex).
\end{proposition}
\begin{proof}
		Let $\bar x \in \bR$. Since $\Phi(\cdot,y_i) -\beta_i \in \mathcal{C}^1(\bR^n)$ and $\mathcal{I}$ is finite, \cite[Exercise 8.31]{RoWe98} gives
		$$
			\OLDpartial f(\bar x) = \con \{r\phi'(\bar x - y_i) : r\phi(\bar x- y_i) - \beta_i = f(\bar x),~i \in \mathcal{I} \}.
		$$
		Let $\bar v \in \OLDpartial f(\bar x)$ be fixed.
		Then, there exist $i_1,i_2 \in \mathcal{I}$ (not necessarily distinct) with $r\phi(\bar x- y_{i_j}) - \beta_{i_j} = f(\bar x)$ and $j \in \{1,2\}$ such that $r\phi'(\bar x - y_{i_1})\leq \bar v\leq r\phi'(\bar x - y_{i_2})$.
		Define \(h\coloneqq\max\set*{ r\phi({}\cdot{} - y_{i_j}) - \beta_{i_j}}[j \in \{1,2\}]\). In view of \cite[Exercise 8.31]{RoWe98} $\bar v \in \OLDpartial h(\bar x)=\con \{r\phi'(\bar x - y_{i_j}) : r\phi(\bar x- y_{i_j}) - \beta_{i_j} = h(\bar x),~j \in \{1,2\} \}$. \cref{thm:subgradient_pointwise_max} applied to $h$ then yields
$$
	h(x)\geq h(\bar x) + r\phi(x-\bar x + (\phi^*)'(r\bar v)) - r\phi((\phi^*)'(r\bar v)) \quad \forall x \in \bR^n.
$$	
		By using the fact that $h(\bar x) = f(\bar x)$ we have for $\bar y\coloneqq\bar x - (\phi^*)'(r\bar v)$

		\[
				f(x) \defeq \max_{i \in \mathcal{I}} \Phi(x,y_i) -\beta_i \geq h(x) \geq f(\bar x) + r\phi(x - \bar y) - r\phi(\bar x - \bar y)
			\quad \forall x \in \bR.
		\]
		Since $(\bar x, \bar v) \in \graph \OLDpartial f$ was arbitrary the claim follows.
\end{proof}

\subsection{Anisotropic strong convexity}\label{sec:gap:a-strong}~
\begin{figure}[t!]
	\includetikz[width=0.49\linewidth]{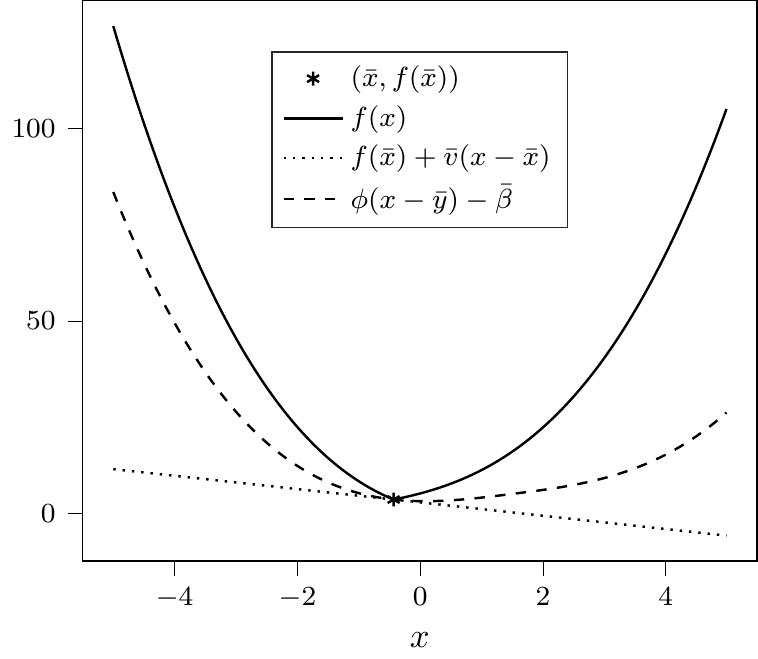}%
	\hfill
	\includetikz[width=0.49\linewidth]{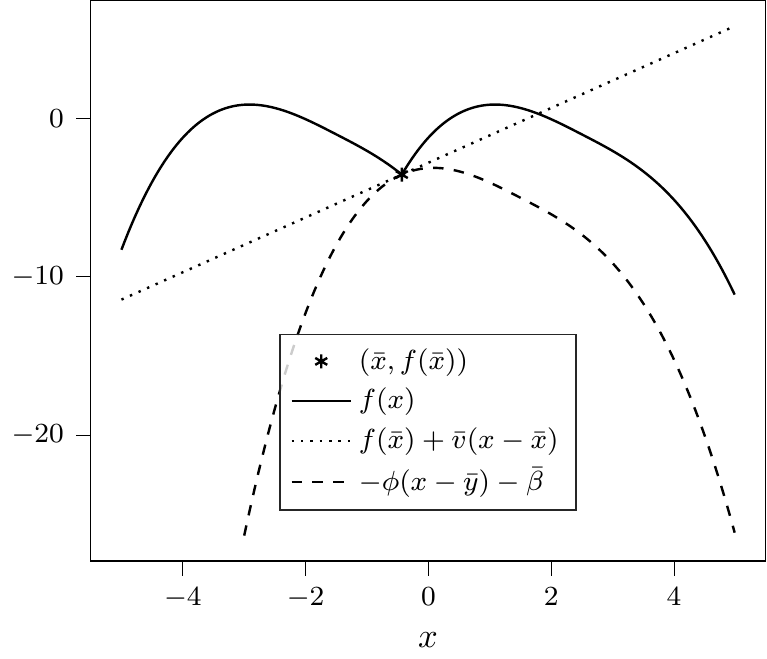}%

	\includetikz[width=0.49\linewidth]{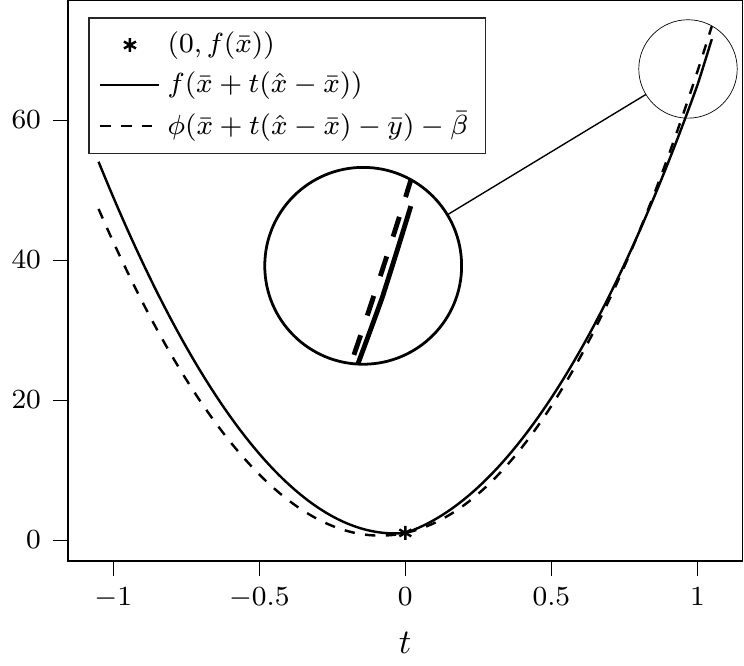}%
	\hfill
	\includetikz[width=0.49\linewidth]{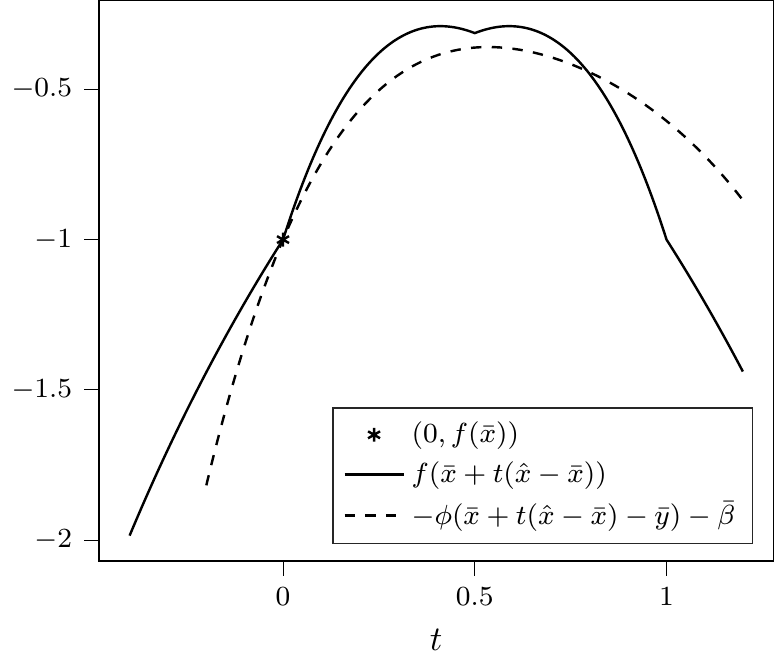}%
	\caption{Upper row: Illustration of \a-strong convexity (left) and \a-weak convexity (right) of univariate pointwise maxima as in \cref{thm:subgradient_pointwise_max}. Lower row: failure of \a-strong convexity (left) and \a-weak convexity (right) of multivariate pointwise maxima plotted along the line $t \mapsto \bar x + t(\hat x - \bar x)$ as described in \cref{ex:counter_rel_str_cvx,ex:counterex_weak_nd}. The dashed lines correspond to the lower approximations $x \mapsto r\phi(x - \bar y) - \bar \beta$ with $\bar y = \bar x - \protect\nabla \phi^*(r\bar v)$ and $\bar \beta = r\phi(\protect\nabla \phi^*(r\bar v)) - f(\bar x)$ as in the anisotropic subgradient inequality \eqref{eq:a-ineq} where $\bar v \in \protect\OLDpartial f(\bar x)$ is an element of the multivalued subdifferential of $f$ at $\bar x$, a point where the graph of $f$ has a cusp. While in the upper row these approximations are global lower bounds, in the lower row the dashed curves are not below the solid ones.
	}%
	\label{fig:lower_bounds_failure}%
\end{figure}
In the univariate case, in light of \cref{thm:subgradient_1d} for $r=+1$ $\Phi$-convexity with finite $\mathcal{I}$ implies \a-strong convexity. The situation is depicted in \cref{fig:lower_bounds_failure} in the upper left.
Since finite-valued $\Phi$-convex functions for $r=+1$ are in particular convex and thus strictly continuous, we can derive the following refinement of \cref{thm:subgradient_1d} for $\Phi$-convex functions with $\mathcal{I}$ not necessarily being finite.
\begin{proposition}[\a-strong convexity of supremal convolutions on \(\R\)]%
\label{thm:astrongly_convex_1d}
	Let $\func{f}{\bR}{\bR}$ be \(\Phi\)-convex for \(\Phi(x,y)=\phi(x-y)\) (equivalently, $f=(-g)\supconv\phi$ for some $g\in\Gamma_0(\R)$).
	Then, $f$ is \a-strongly convex.
\end{proposition}
\begin{proof}
	Since $f$ is finite-valued, in view of \cite[Example 9.14]{RoWe98} $f$ is strictly continuous on $\bR$.
	By \cref{thm:Phicvx} we have for any $x \in \bR$
		\begin{equation}\label{eq:lowerC1}
			f(x) = \biconj f(x) = \sup_{y\in\R}\phi(x - y) - \conj f(y),
		\end{equation}
		and by \cref{thm:phi_subgrad_existence} it holds that $y \in \subdiff f(x) \Leftrightarrow x \in \subdiff \conj f(y) \Leftrightarrow y \in \argmax \{\phi(x-\cdot) - \conj f\}$. In light of \cref{thm:subset_phi_resolvent_inv,thm:phi_subdiff_nonempty} $\emptyset \neq \subdiff f(x)\subseteq x - (\phi^*)'(\OLDpartial f(x))$
		and hence the supremum in \cref{eq:lowerC1} is attained on $\subdiff f(x)\subseteq x - (\phi^*)'(\OLDpartial f(x))$. 
		Let $\bar x \in \bR$. Because of strict continuity, $\OLDpartial f$ is locally bounded at $\bar x$, and as such so is $\subdiff f$.
		Thus there exists a compact set $Y$ so that for any $x$ sufficiently near $\bar x$ we can restrict the maximization over $y$ in \eqref{eq:lowerC1} to $Y$, i.e., $f(x)=\sup_{y\in Y}\phi(x - y) - \conj f(y)$ for $x$ near $\bar x$. This implies that $f$ is lower-$\mathcal{C}^1$ in the sense of \cite[Definition 10.29]{RoWe98}. In light of \cite[Theorem 10.31]{RoWe98} we have that
		$$
\OLDpartial f(\bar x) \subseteq \con \{\phi'(\bar x - y) : y \in \subdiff f(\bar x) \}.
$$
Let $\bar v \in \OLDpartial f(\bar x)$ be fixed. Then there exist $y_1, y_2 \in \subdiff f(\bar x)=\argmax \{\phi(\bar x-\cdot) - \conj f\}$ (not necessarily distinct) such that $\phi'(\bar x - y_1)\leq \bar v\leq \phi'(\bar x - y_2)$. In particular we have $\phi(\bar x- y_{i}) - \conj f(y_i) = f(\bar x)$ for $i \in \{1,2\}$. Define \(h\coloneqq\max\set*{ \phi({}\cdot{} - y_i) - \conj f(y_i)}[i \in \{1,2\}]\). In view of \cite[Exercise 8.31]{RoWe98} $\bar v \in \OLDpartial h(\bar x)=\con \{\phi'(\bar x - y_{i}) : \phi(\bar x- y_{i}) - \conj f(y_i) = h(\bar x),~i \in \{1,2\} \}$. \cref{thm:subgradient_pointwise_max} applied to $h$ then yields
$$
	h(x)\geq h(\bar x) + \phi(x-\bar x + (\phi^*)'(\bar v)) - \phi((\phi^*)'(\bar v)) \quad \forall x \in \bR^n.
$$		
		By using the fact that $\phi(\bar x- y_{i}) - \conj f(y_i) = f(\bar x)$ for $i \in \{1,2\}$ and thus $h(\bar x) = f(\bar x)$ we have for $\bar y\coloneqq\bar x - (\phi^*)'(\bar v)$
		\[
				f(x) = \biconj f(x) = \sup_{y\in\R}\phi(x - y) - \conj f(y) \geq h(x) \geq f(\bar x) + \phi(x - \bar y) - \phi(\bar x - \bar y)
			\quad \forall x \in \bR.
		\]
		Since $(\bar x, \bar v) \in \graph \OLDpartial f$ was arbitrary the claim follows.
\end{proof}
In light of \cref{thm:duality_aniso_str_cvx} this shows that the saddle-point property \eqref{eq:strong_duality_aniso_str_cvx} holds automatically in the finite-valued univariate case.

In the multivariate case, however, the situation is different: $\Phi$-convex functions on $\bR^n$ in general fail to be \a-strongly convex.
This is verified by the following counterexample

\begin{example}[failure of \a-strong convexity in higher dimensions]\label{ex:counter_rel_str_cvx}%
	Choose $\func{\phi}{\bR^2}{\bR}$ as $\phi(x)=x_1^2 + x_2^4$, and for $y=(-1,1)$ let
	$$
		f(x)\coloneqq\max\set{\phi(x), \phi(x-y)}.
	$$
	The graphs of the two pieces $\phi, \phi(\cdot - y)$ intersect at points $x \in \bR^2$ at which $\phi(x)=\phi(x-y)$, i.e., which satisfy $x_1 =2x_2^3 -3x_2^2 + 2x_2 -1$. At these points the graph of $f$ has a downwards pointing cusp. 
	We have $\nabla \phi(x)=(2x_1, 4x_2^3)$ and $\nabla \phi(x-y)=(2x_1 + 2, 4(x_2 - 1)^3)$.
	Take $\bar x=(-1,0)$, at which $\phi(\bar x)=\phi(\bar x-y)$ holds in particular, and
	$$\textstyle
		\bar v
	{}\coloneqq{}
		\binom{-1}{-2}\in\OLDpartial f(\bar x)
	{}={}
		\con\set{\binom{-2}{0},~ \binom{0}{-4}},
	$$
	so that $\nabla \phi(\bar x - \bar y) = \bar v$ and $\phi(\bar x - \bar y) - \bar \beta = f(\bar x)$
	for
	\(
		\bar y = \binom{-\nicefrac12}{2^{-1/3}}
	\)
	and
	\(
		\bar\beta = \phi(\bar x - \bar y) - f(\bar x) = \tfrac14 + 2^{-4/3} -1
	\).
	At $\hat x=(-8, -1)$ we have $f(\hat x) = 65$ and $\phi(\hat x - \bar y) -\bar \beta 
	> 65=f(\hat x)$, and thus $\phi(\cdot - \bar y) -\bar \beta$ is not a lower bound of $f$. As a consequence, $f$ is not \a-strongly convex. This is shown in \cref{fig:lower_bounds_failure} in the lower left: it can be seen that the dashed line, the graph of $\phi(\cdot - \bar y) -\bar \beta$, is above the solid line, the graph of $f$ when restricted to the line $\bar x + t(\hat x - \bar x)$ and $t$ is close to $1$.
\end{example}
In light of \cref{thm:duality_aniso_str_cvx} this shows that the saddle-point property \eqref{eq:strong_duality_aniso_str_cvx} does not hold in general.

\subsection{Anisotropic weak convexity}
In light of \cref{thm:subgradient_1d}, for $r=-1$, univariate pointwise maxima over finite collections of functions $\Phi(\cdot, y_i) -\beta_i$ are \a-weakly convex.

The following example reveals that this is no longer true in the multivariate case.

\begin{example}[failure of \a-weak convexity in higher dimensions] \label{ex:counterex_weak_nd}%
	Choose $\func{\phi}{\bR^2}{\bR}$ as $\phi(x)=x_1^2 + x_2^4$, and for $y=(1,-1)$ let
	$$
		f(x)
	{}\coloneqq{}
		\max\set{-\phi(x), -\phi(x-y)}
	{}={}
		\begin{ifcases}
			-\phi(x) & x_1 \leq 2x_2^3 + 3x_2^2 + 2x_2 + 1\\
			-\phi(x-y)\otherwise.
		\end{ifcases}
	$$
	The graphs of the two pieces $-\phi, -\phi(\cdot - y)$ intersect at points $x \in \bR^2$ at which $\phi(x)=\phi(x-y)$.
	Take $\bar x=(1,0)$, for which this holds in particular, and
	$$\textstyle
		\bar v
	{}\coloneqq{}
		\binom{-1}{-2}\in\OLDpartial f(\bar x)
	{}={}
		\con\set{\binom{-2}{0},~ \binom{0}{-4}},
	$$
	so that $-\nabla \phi(\bar x - \bar y) = \bar v$ and $-\phi(\bar x - \bar y) - \bar \beta = f(\bar x)$ for
	$$\textstyle
		\bar y = \binom{1/2}{-2^{-1/3}}
	\quad\text{and}\quad
		\bar\beta = \phi(\bar x - \bar y) - f(\bar x) = \frac{3-2^{2/3}}{4}.
	$$
	At $\hat x=(0, -1)$ we have $-\phi(\hat x - \bar y) -\bar \beta = 2^{1/3}(2^{4/3}-3) > -1=f(\hat x)$, and thus $-\phi(\cdot - \bar y) -\bar \beta$ is not a lower bound of $f$. As a consequence, $f$ is not \a-weakly convex.
		This is shown in \cref{fig:lower_bounds_failure} in the lower right: it can be seen that the dashed line, the graph of $-\phi(\cdot - \bar y) -\bar \beta$, is above the solid line, the graph of $f$ when restricted to the line $\bar x + t(\hat x - \bar x)$ and $t$ is close to $1$.
\end{example}

	\section{Conclusion}\label{sec:conclusion}%
		In this paper we provided the dual counterparts to relative smoothness and strong convexity (\B-smoothness and \B-strong convexity), called anisotropic smoothness and strong convexity (\a*-smoothness and \a*-strong convexity).
In the context of generalized convexity, both notions can be seen as $\Phi$-subgradient inequalities, or as anisotropic generalizations of the descent lemma and strong convexity inequality which can be fruitful in the development of first-order algorithms.
Unlike the Euclidean case, we have shown a gap between the classes of functions that satisfy the anisotropic strong convexity subgradient inequality and the class of supremal convolutions, the latter being a proper subset of the former
for which a certain saddle-point property holds. In spite of the duality relation being limited to convex functions, the viewpoint of $\Phi$-convexity naturally motivated the study of \emph{weak convexity} counterparts which led to the discovery of yet another unexpected gap. While \B-weak convexity can be characterized in terms of $\Phi$-convexity, there exist \a-weakly convex functions that are not $\Phi$-convex (with the corresponding $\Phi$); whether a converse inclusion holds is an open problem. It is also unclear whether differentiability follows from \a-weak convexity of both $f$ and $-f$, in the same way it does for the Bregman notions. In general, we believe that a more detailed investigation of anisotropic weak convexity is an interesting direction for future research. Algorithmic implications and more practical examples of the newly introduced anisotropic notions are under investigation in \cite{laude2022anisotropic}.

	\ifsiam
		\bibliographystyle{siamplain}
	\else
		\phantomsection
		\addcontentsline{toc}{section}{References}%
		\bibliographystyle{abbrv}
	\fi
	\bibliography{references.bib}
\end{document}